\documentclass[12pt,reqno]{amsart}%
\usepackage{amsmath, amsfonts, amssymb, amsthm, amscd, amsbsy}
\usepackage[usenames,dvipsnames,svgnames,x11names,hyperref]{xcolor}
\usepackage[tmargin=0.8in,bmargin=0.8in,lmargin=0.9in,rmargin=0.9in]{geometry}
\usepackage{enumerate}
\usepackage{xcolor}
\usepackage[pagebackref]{hyperref}%
\usepackage{amsmath}%
\usepackage{bm}
\usepackage{amsfonts}%
\usepackage{amssymb}
\usepackage{mathrsfs}
\usepackage[numbers,sort]{natbib}

\providecommand{\U}[1]{\protect\rule{.1in}{.1in}}
\hypersetup{backref=true,
  pagebackref=true,
  hyperindex=true,
  colorlinks=true,
  breaklinks=true,
  urlcolor=NavyBlue,
  linkcolor=Fuchsia,
  bookmarks=true,
  bookmarksopen=false,
  filecolor=black,
  citecolor=ForestGreen,
  linkbordercolor=red
}

\newtheorem{theorem}{Theorem}[section]
\theoremstyle{plain}

\newtheorem{corollary}[theorem]{Corollary}

\newtheorem{lemma}[theorem]{Lemma}

\newtheorem{proposition}[theorem]{Proposition}
\newtheorem{remark}[theorem]{Remark}

\numberwithin{equation}{section}

\newtheorem{theoremalph}{Theorem}

\theoremstyle{definition}

\allowdisplaybreaks

\def\a{\alpha}

\def\d{\delta}
\def\e{\epsilon}
\def\la{\lambda}
\def\p{\partial}
\def\g{\gamma}
\def\grad{\nabla}
\def\oo{\infty}

\def\Ric{\operatorname{Ric}}

\def\R{\mathbb{R}}

\def\B{\mathbb{B}}

\def\H{\mathbb{H}}
\def\N{\mathbb{N}}

\def\mcQ{\mathcal{Q}}
\def\mcE{\mathcal{E}}

\def\mcC{\mathcal{C}}

\def\mcM{\mathcal{M}}
\def\mcP{\mathcal{P}}

\newcommand{\gen}[1]{\ensuremath{\langle #1\rangle}}

\def\ceven{C_{\operatorname{even}}^{\oo}}
\def\CHspace{\mcC^{2\g}(X) \cap \dot H^{k,\g}(X)}

\begin{document}
\title[Boundary Operators and Trace Inequalities]{Conformally Covariant Boundary Operators and Sharp Higher Order Sobolev Trace Inequalities on Poincar\'e-Einstein Manifolds}
\author{Joshua Flynn}
\address{Joshua Flynn: CRM/ISM and McGill University\\
Montr\'{e}al, QC H3A0G4, Canada}
\email{joshua.flynn@mcgill.ca}
\author{Guozhen Lu }
\address{Guozhen Lu: Department of Mathematics\\
University of Connecticut\\
Storrs, CT 06269, USA}
\email{guozhen.lu@uconn.edu}
\author{Qiaohua Yang}
\address{School of Mathematics and Statistics, Wuhan University, Wuhan, 430072, People's  Republic of China}
\email{qhyang.math@whu.edu.cn}

\thanks{The first two authors were partially supported by a grant from the Simons Foundation. The third author was  partially supported by the National Natural Science Foundation of China (No.12071353)}
  \textbf{ }

\begin{abstract}
  In this paper we introduce conformally covariant boundary operators for Poincar\'e-Einstein manifolds satisfying a mild spectral assumption.
  Using these boundary operators we set up higher order Dirichlet problems whose solutions are such that, when applied to by our boundary operators, they recover the fractional order GJMS operators on the conformal infinity of the manifold.
  We moreover obtain all related higher order trace Sobolev inequalities on these manifolds.
  In conjunction with Beckner's fractional Sobolev inequalities on the sphere, we obtain as an application the sharp higher order Sobolev trace inequalities on the ball.
\end{abstract}

\maketitle
\tableofcontents

\section{Introduction and statement of results}

The fractional Laplacian $(-\Delta_{x})^{\g}$, $\g \in (0,\oo)$, on $\R^{n}$ is a nonlocal pseudodifferential operator which may be defined via the Fourier transform via
\[
  \widehat{(-\Delta)_{x}^{\g}} f(\xi) = |\xi|^{2\g} \hat f(\xi).
\]
Among many other ways to define $(-\Delta_{x})^{\g}$, one may also define $(-\Delta_{x})^{\g}$ as a singular integral:
\[
  (-\Delta_{x})^{\g} f(x) = c \int_{\R^{n}} \frac{f(x) - f(y)}{|x-y|^{n+2\g}} dy,
\]
for some suitable normalization constant $c$.
Through the seminal work in \cite{MR2354493} Caffarelli-Silvestre established that, for $\g \in (0,1)$, one may obtain $(-\Delta_{x})^{\g}$ via a Dirichlet-to-Neumann mapping.
Namely, if $U$ is a solution to
\begin{equation}
  \begin{cases}
     (\Delta + (1-2\g)y^{-1}\p_{y}) U = 0 & \text{ in }\R_{+}^{n+1}\\
    U = f &  \text{ on }\R^{n}
  \end{cases},
  \label{eq:2nd-order-dirichlet}
\end{equation}
where $\Delta_{m} : = \Delta + m y^{-1} \p_{y}$, $m= 1-2\g$, and $\R_{+}^{n+1} := \R^{n} \times (0,\oo)$ is the halfspace, then
\begin{equation}
  (-\Delta_{x})^{\g} f = - 2^{2\g-1}\frac{\Gamma(\g)}{\Gamma(1-\g)}\lim_{y \to 0 } y^{1-2\g}\p_{y}U.
  \label{eq:first-order-boundary-operator}
\end{equation}
Therefore, one may obtain $(-\Delta_{x})^{\g}$ by means of a boundary operator mapping functions on $\overline{\R_{+}^{n+1}}$ to functions on $\R^{n}$.
Moreover, by the Dirichlet principle, there also holds a sharp trace inequality relating functions $U(x,y)$ to $(-\Delta_{x})^{\g}$ acting on its boundary data $f(x) = U(x,0)$; i.e.,
\begin{equation}
  2^{1-2\g} \frac{\Gamma(1-\g)}{\Gamma(\g)} \int_{\R^{n}} f (-\Delta_{x})^{\g}f dx \leq \int_{\R_{+}^{n+1}} |\grad U|^{2} y^{1-2\g} dxdy
  \label{eq:first-order-trace-inequality}
\end{equation}
and equality holds iff $U$ solves \eqref{eq:2nd-order-dirichlet}.
Using Lieb's fractional Sobolev inequality from \cite{MR717827}, one obtains the sharp Sobolev trace inequality: for $U \in W^{1,2}(\R_{+}^{n+1},y^{1+2\g})$ and $f(x) := U(x,0)$, there holds
\begin{equation}
  2^{1-2\g}\frac{\Gamma(1-\g)}{\Gamma(\g)} \frac{\Gamma(\frac{n+2\g}{2})}{\Gamma(\frac{n-2\g}{2})}Vol(S^{n})^{\frac{2\g}{n}} \left( \int_{\R^{n}} |f|^{\frac{2n}{n-2\g}} dx \right)^{\frac{n-2\g}{n}} \leq \int_{\R^{n+1}_{+}} |\grad U|^{2} y^{1-2\g} dx dy,
  \label{eq:first-order-sharp-sobolev-trace-inequality}
\end{equation}
with equality if and only if $U$ satisfies \eqref{eq:2nd-order-dirichlet} with $f(x) = a(\e + |x-\xi|^{2})^{-\frac{n-2\g}{2}}$ for some $a \in \R$, $\e>0$ and $\xi \in \R^{n}$.

Since the fractional powers $(-\Delta_{x})^{\g}$ can be defined for $\g \in (0,\oo) \setminus \N$, it is natural to ask whether the Caffarelli-Silvestre extension theorem and a corresponding trace inequality holds for higher fractional orders.
Indeed, such higher order Caffarelli-Silvestre's extension theorem and corresponding trace inequalities on $\R_{+}^{n+1}$ were established  by Case in \cite{MR4095805} by introducing suitable higher order boundary operators. Case showed that these boundary operators map solutions of a higher order Dirichlet problem (generalizing \eqref{eq:2nd-order-dirichlet}) to fractional powers of $-\Delta_{x}$ acting on the boundary data.
More precisely, Case introduced the following boundary operators: letting $\g \in (0,\oo) \setminus \N$ and  $m = 1-2[\g]$, then, on a space $\mcC^{2\g}(\R_{+}^{n+1})$ of functions which have suitable asymptotic expansions near $y=0$, the boundary operators $B_{2j}^{2\g}:\mcC^{2\g}(\R_{+}^{n+1}) \to C^{\oo}(\R^{n})$ and $B_{2[\g]+2j}^{2\g}:\mcC^{2\g}(\R_{+}^{n+1}) \to C^{\oo}(\R^{n})$ are recursively defined for $0\leq j \leq \lfloor \g \rfloor$ by
\begin{equation}
  \begin{aligned}
  B_{2j}^{2\g} &= (-1)^{j} \iota^{*} \circ T^{j} - \sum_{\ell=1}^{j} {j \choose \ell} \frac{\Gamma(1+j-[\g])\Gamma(1+2j-2\ell-\g)}{\Gamma(1+j-\ell-[\g])\Gamma(1+2j-\ell-\g)}\Delta_{x}^{\ell} B_{2j-2\ell}^{2\g}\\
  B_{2[\g]+2j}^{2\g} &= (-1)^{j+1} \iota^{*} \circ y^{m} \p_{y} T^{j}\\
  &- \sum_{\ell=1}^{j} {j \choose \ell} \frac{\Gamma(1+j+[\g])\Gamma(1 + 2j - 2\ell - \lfloor \g \rfloor + [\g])}{\Gamma(1+j-\ell+[\g])\Gamma(1+2j-\ell-\lfloor \g \rfloor + [\g])} \Delta_{x}^{\ell} B_{2[\g]+2j-2\ell}^{2\g}.
    \end{aligned}
  \label{eq:case-boundary-operator}
\end{equation}
where $T=\p_{y}^{2} + m y^{-1}\p_{y}$ and the empty sum equals zero by convention.
Here, $\Delta_{x}$ is the Laplacian on $\R^{n}$ and $\iota^{*}:C^{\oo}(\overline{\R_{+}^{n+1}}) \to C^{\oo}(\R^{n})$ is the restriction operator mapping functions $U(x,y)$ to $U(x,0)$.
Define also the related Dirichlet form
\begin{align*}
  \mcQ_{2\g}(U,V) &= \int_{\R_{+}^{n+1}} U L_{2k} V y^{m} dxdy\\
  &+\sum_{j=0}^{\lfloor \g/2 \rfloor} \int_{\R^{n}} B_{2j}^{2\g}U B_{2\g-2j}^{2\g}V dx - \sum_{j=0}^{\lfloor \g \rfloor - \lfloor \g/2 \rfloor - 1} \int_{\R^{n}} B_{2[\g] + 2j}^{2\g}U B_{2\lfloor \g \rfloor -2j}^{2\g} V dx
\end{align*}
and energy
\[
  \mcE_{2\g}(U) = \mcQ_{2\g}(U,U),
\]
where $U, V \in \mcC^{2\g}(\R_{+}^{n+1})$ and $L_{2k} = \Delta_{m}^{k}$.
Then Case showed that a specific $B_{2\g}^{2\g}$ applied to a solution $U$ to the underdetermined problem $\Delta^{k+1}U =0$ recovers $(-\Delta_{x})^{\g}\iota^{*}U$ thereby extending \eqref{eq:first-order-boundary-operator} to the higher order setting.
Case also proved in  \cite{MR4095805} the following sharp Sobolev trace inequalities on $\R_{+}^{n+1}$.

\begin{theoremalph}
  There holds
  \begin{align*}
    \mcE_{2\g}(U) &\geq \sum_{j=0}^{\lfloor \g/2 \rfloor} c_{\g,j} \int_{\R^{n}} f^{(2j)}(-\Delta_{x})^{\g}f^{(2j)}dx\\
    &+ \sum_{j=0}^{\lfloor \g \rfloor - \lfloor \g/2 \rfloor - 1}d_{\g,j} \int_{\R^{n}} \phi^{(2j)}(-\Delta_{x})^{\lfloor \g \rfloor - [\g] -2j}\phi^{(2j)} dx
  \end{align*}
  for all $U \in \mcC^{2\g}(\R_{+}^{n+1}) \cap W^{\lfloor \g \rfloor + 1, 2}(\R_{+}^{n+1},y^{1-2[\g]})$, where $f^{(2j)} = B_{2j}^{2\g}U$, $0 \leq j \leq \lfloor \g /2 \rfloor$ and $\phi^{(2j)} = B_{2[\g]+2j}^{2\g}U$, $0 \leq j \leq \lfloor \g \rfloor - \lfloor \g/2 \rfloor - 1$, and the constants $c_{\g,j}$ and $d_{\g,j}$ are sharp and explicit.
  Moreover, equality holds iff $L_{2k}U=0$.
\end{theoremalph}

It is important to contrast these results with existing ones in the literature.
For example, in \cite{MR2737789}, Chang-Gonz\'alez obtained through an induction argument and Graham-Zworski scattering theory \cite{MR1965361} all of the higher order extension results.
More precisely, they established the following theorem.
\begin{theoremalph}
  \label{thm:chang-gonzalez-theorem}
  Fix $\g \in (0,\frac{n}{2}) \setminus \N$, $ a= 1 - 2\g$.
  Let $f$ be a smooth function defined on $\R^{n}$, and let $U = U(x,y)$ be a solution to the boundary value problem
  \[
    \begin{cases}
      \Delta_{x} U + \frac{a}{y}\p_{y} U + \p_{yy} U =0 &\text{ in } \R_{+}^{n+1}\\
      U(x,0) = f(x) & \text{ on } \R^{n}.
    \end{cases}
  \]
  The function $u = y^{n-s}U$ is a solution of the Poisson equation on the hyperbolic space $(\H^{n+1},g_{\H})$
  \[
    \begin{cases}
      -\Delta_{g_{\H}}u - s(n-s)u =0 & \text{ in } \H^{n+1}\\
      u = Fy^{n-s} + H y^{s}\\
      F(x,0) = f(x).
    \end{cases}
  \]
  And, the following limit exists and we have the equality
  \[
    (-\Delta_{x})^{\g} f = \frac{d_{\g}}{2\g_{0}} A_{m}^{-1} \lim_{y\to0} y^{a_{0}} \p_{y} [y^{-1} \p_{y}(y^{-1}\p_{y} (\cdots y^{-1} \p_{y}U ))],
  \]
  where we are taking $m+1$ derivatives in the above expression, and the constant is given by
  \[
    A_{m} = 2^{m}(\g-1)\cdots (\g-m+1).
  \]
  We are using the notation $m=\lfloor \g \rfloor$, $\g_{0} = \g-m$, $a_{0}  = 1 - 2\g_{0}$, and
  \[
    d_{\g} = 2^{\g} \frac{\Gamma(\g)}{\Gamma(-\g)}.
  \]
\end{theoremalph}
While Theorem \ref{thm:chang-gonzalez-theorem} clearly extends Caffarelli-Silvestre's extension result, one does not have at their disposal a Dirichlet principle to obtain the higher order trace inequalities which extend \eqref{eq:first-order-trace-inequality}.
On the other hand, in \cite{MR3592161,yang2013higher}, where, instead of considering a second order Dirichlet problem, Chang-Yang established higher order extension results with corresponding sharp trace inequalities  with extra Neumann-type boundary conditions in order for the trace inequalities to hold.
This was done in part by considering solutions to a higher order determined Dirichlet problem.
The work of Case in \cite{MR4095805} therefore established a complete higher order Caffarelli-Silvestre-type extension theory on $\R_{+}^{n+1}$ with corresponding sharp higher order trace inequalities.

We also mention the results of Beckner \cite{MR1230930} and  Ache-Chang in \cite{MR3707288}, and subsequent results of  Ngo, Nguyen and Phan \cite{MR4053619},
 and Q. Yang \cite{YangQ} where they established the sharp trace Sobolev inequality of higher order on the real unit ball $\B^{n+1} \subset \R^{n+1}$, respectively.
In particular in \cite{MR3707288}, they used this inequality to characterize the extremal metric of the main term in the log-determinant formula corresponding to the conformal Laplacian coupled with the boundary Robin operator on $\B^{n+1}$; see \cite{MR1454485,MR1454486}.
In their paper, Ache-Chang relied on the scattering theory on $\B^{n+1}$ and appropriate choice of ``adapted metric;'' see \cite{MR3493624} for where Case-Chang introduced the adapted metric, which is a natural metric for the study of Sobolev inequalities (e.g., see \cite{MR3707288,MR3743194,MR3592161}).
See also \cite{MR4053619,MR4285731} for more lower order results.
For other  Hardy-Adams type inequalities on the unit ball, see \cite{MR3818080,MR4082245,MR3695883,MR4030527,MR4388948,FLY23quaternion}.

Many of the aforementioned results on extensions of Caffarelli-Silvestre's results rely heavily on the Graham-Zworski scattering theory \cite{MR1965361} on real hyperbolic space.
Since such scattering theory and fractional operators are available to more general Poincar\'e-Einstein manifolds, it is natural and geometrically significant to investigate generalizing the Caffarelli-Silvestre extension theory and higher order trace inequalities in more general geometries.
Indeed, the extension result in Theorem \ref{thm:chang-gonzalez-theorem} was established by Chang-Gonz\'alez (also in \cite{MR2737789}) for general conformally compactifiable Einstein manifolds.
However,  the sharp higher order Sobolev trace inequalities are still missing in this context.
It is therefore  not clear how to suitably define the boundary operators \eqref{eq:case-boundary-operator} of Case to more general Poincar\'e-Einstein manifolds.
Indeed, in \cite{MR3619870}, Case considered the situation  $\g \in (0,1) \cup (1,2)$ where they introduced suitable lower order boundary operators on general Poincar\'e-Einstein manifolds to develop the extension theory and establish the corresponding trace inequalities for this range of $\g$.
 
In this paper, we establish the complete generalization of the Caffarelli-Silvestre extension theorem and Sobolev trace inequalities for $\g \in (0,\oo) \setminus \N$ to the Poincar\'e-Einstein setting with a mild spectral assumption (the same kind of assumption as assumed by Case in \cite{MR3619870}).
As an application, we also obtain these results on the unit ball $\B^{n+1} \subset \R^{n+1}$ and, through conformal covariance, also obtain such results on $\R_{+}^{n+1}$.
To achieve this, we introduce natural and simply written conformally covariant boundary operators.
For a suitable range of the indices, our boundary operators take the following form:
\begin{align*}
  B_{2j}^{2\g } &= b_{2j}\rho^{-\frac{n}{2} + \g - 2j}  \circ \prod_{\ell=0}^{j-1}\left( \Delta_{+} + \frac{n^{2}}{4} - (\g -2\ell)^{2} \right)\\
  &\circ \prod_{\ell=0}^{j-1} \left( \Delta_{+} + \frac{n^{2}}{4} - (\g + 2\ell - 2\lfloor \g \rfloor )^{2} \right) \circ  \rho^{\frac{n}{2}-\g } |_{\rho=0}\\
  B_{2j + 2[\g]}^{2\g}
  &= b_{2j+2[\g]}\rho^{-\frac{n}{2} + \g - 2j  - 2[\g]}\circ \prod_{\ell=0}^{j} \left( \Delta_{+} + \frac{n^{2}}{4} - (\g -2\ell)^{2} \right)\\
  &\circ\prod_{\ell=0}^{j - 1} \left( \Delta_{+} + \frac{n^{2}}{4} - (\g  + 2\ell - 2\lfloor \g \rfloor )^{2} \right) \circ \rho^{\frac{n}{2}-\g}  |_{\rho=0},
\end{align*}
where $b_{2j},b_{2j+2[\g]}$ are optimal and explicit constants, where $\Delta_{+}$ is the Laplace-Beltrami operator on the given Poincar\'e-Einstein manifold $(X^{n+1},M^{n},g_{+})$ and where $\rho$ is any defining function with suitable asymptotics near the boundary $M$ of $X$.
See Section \ref{sec:boundary-operators} for exact definitions and for the boundary operators defined for the full range of indices.

To be more precise, recall $(X^{n+1},M^{n},g_{+})$ is a Poincar\'e-Einstein manifold provided $X$ is diffeomorphic to the interior of a compact manifold $\overline{X^{n+1}}$ with boundary $M$ and $g_{+}$ satisfies suitable asymptotics near $M$ (see Section \ref{sec:preliminaries} for exact definitions).
In particular, there is a defining function $\rho$ such that $\rho^{2}g_{+}$ extends smoothly onto $\overline{X}$ and this induces a conformal structure $[g]$ on $M$.
To each representative $g \in [g]$, there is a unique defining function $r$, called the geodesic defining function, such that
\[
  g_{+} = \frac{dr^{2} + h_{r}}{r^{2}}  \quad \text{ near } \quad M,
\]
where the $h_{r}$ are metrics on $M$ and $h_{0} = g$.
As is well-known, the GJMS operators and their fractional analogues (e.g., see \cite{MR3493624}) may be obtained through the scattering theory of Graham-Zworski \cite{MR1965361}, which we briefly recall now.
Fix $\g \in (0,\frac{n}{2})\setminus\N$ and set $k = \lfloor \g \rfloor + 1$.
Let $\Delta_{+}$ denote the Laplace-Beltrami operator for $g_{+}$ and suppose $\frac{n^{2}}{4} - \g^{2}$ is not in the $L^{2}$-spectrum of $-\Delta_{+}$.
Given $f \in C^{\oo}(M)$ and $s = \frac{n}{2} + \g$, there is a unique solution $\mathcal{P}(s)f$ of the Poisson equation
\begin{equation}
  \Delta_{+} V + s(n-s)V = 0
  \label{eq:scattering-equation-intro}
\end{equation}
such that, asymptotically near $M$, there holds
\begin{equation}
  \begin{cases}
    \mathcal{P}(s)f =  r^{n-s}F +  r^{s}G & \text{ for some } F,G \in C^{\oo}(\overline{X})\\
    F|_{M} = f\\
  \end{cases}.
  \label{eq:scattering-equation-expansion-intro}
\end{equation}
Defining the scattering operator $S(s):F|_{M}  \mapsto  G|_{M}$, we recall that the operator $S$ defines the fractional GJMS operators $P_{\g}$ on $M$ by
\begin{equation}
  P_{\g} f =c_{\g}S(s)f, \quad  c_{\g} = 2^{\g} \frac{\Gamma(\g)}{\Gamma(-\g)}.
  \label{eq:fractional-boundary-operator-intro}
\end{equation}
Note that analytic continuation allows one to recover the integral order GJMS operators.
The operators $P_{\g}$ are formally self-adjoint, have principle symbol $(-\Delta)^{\g}$ and are conformally covariant in the sense that, if $\hat g = e^{2\tau}g$ is conformal to the given $g \in [g]$ for some $\tau \in C^{\oo}(M)$, then
\[
  \hat{P}_{\g} f =  e^{(-\frac{n}{2}+ \g)\tau} P_{\g} \left( e^{(\frac{n}{2}-\g)\tau}f \right), \quad f \in C^{\oo}(M).
\]

Throughout, we let $\rho$ always mean a chosen $\g$-admissible defining function (see \eqref{eq:gamma-admissible-defining-function}), which just means it is asymptotic to $r$ in an appropriate way.
Moreover, to make the exposition neater, we will refrain from decorating operators with $\rho$ despite their dependencies on the choice of $\rho$.
Now, motivated by the scattering theory, we define the following weighted operators:
\[
  D_{s} = -\Delta_{+} - s(n-s), \quad L_{2k}^{+} = \prod_{j=0}^{k-1}D_{s-j}, \quad s = \frac{n}{2} + \g
\]
and
\[
  L_{2k} = \rho^{-\frac{n}{2} + \g - 2k} \circ L_{2k}^{+} \circ \rho^{\frac{n}{2} - \g}.
\]
The $L_{2k}$ are naturally conformally covariant in the sense that, if $\hat \rho = e^{\tau} \rho$, then
\[
  \hat L_{2k} = e^{(-\frac{n}{2} + \g - 2k)\tau} \circ L_{2k} \circ e^{(\frac{n}{2} - \g)\tau}.
\]
These operators have been  considered by  Chang and Case in  \cite{MR3493624}.
The point is that our boundary operators $B_{\a}^{2\g}$ are such that, for solutions to $L_{2k}U=0$, specific $B_{\a}^{2\g}$ applied to $U$ recover the fractional operators $P_{\g}$ on the conformal infinity $M$ acting on the boundary data of $U$; this is recorded in Theorem \ref{thm:boundary-to-fractional-operator} below.

We now state our results more explicitly, which will demonstrate that our boundary operators are natural for obtaining the fractional GJMS operators on $M$, as well as the corresponding sharp higher order trace inequalities.
The first result is the conformal covariance property of our boundary operators.
Given another $\g$-admissible defining function $\widehat{\rho}=e^{\tau}\rho$ for some $\tau \in C^{\oo}(\overline{X})$, and if we let $\widehat{B}_{2j}^{2\gamma}$ and $\widehat{B}^{2\g}_{2j+2[\gamma]}$ be the boundary operators associated with $(X, \widehat{\rho}^{2}g_{+})$, then we have the following theorem.
\begin{theorem}\label{thm:boundary-operator-conformal-covariance}
  Let $U\in \mathcal{C}^{2\gamma}$ and $0 \leq j \leq \lfloor \g \rfloor$.
  It holds that
  \begin{align}
    \widehat{B}_{2j}^{2\gamma}U=&e^{(-\frac{n}{2} + \g  - 2j)\tau|_{M}}B_{2j}^{2\gamma}(e^{(\frac{n}{2}-\gamma)\tau}U);\\
    \widehat{B}_{2j+2[\gamma]}^{2\gamma}U=&e^{(-\frac{n}{2} + \g - 2j  - 2[\g])\tau|_{M}}B_{2j+2[\g]}^{2\gamma}(e^{(\frac{n}{2}-\gamma)\tau}U),
  \end{align}
  where $\tau|_{M}$ is the restriction of $\tau$ on $M$.
\end{theorem}

The second result concerns the Caffarelli-Silvestre-type extension properties of our boundary operators.
That is to say, our boundary operators recover the conformally covariant operators $P_{\g}$ on $M$ in the following way.
\begin{theorem}\label{thm:boundary-to-fractional-operator}
  Let  $\gamma\in (0,\frac{n}{2})\setminus\mathbb{N}$ and $k = \lfloor \g \rfloor + 1$.
  Let $(X^{n+1},M^{n},g_{+})$ be a Poincar\'e-Einstein manifold satisfying
  \[
    \frac{n^{2}}{4} - (2\ell-\g)^{2} \notin \sigma_{pp}(\Delta_{g_{+}})\quad  \text{ for }\quad 0 \leq \ell \leq \lfloor \g \rfloor.
  \]
  Let $\rho$ be a $\g$-admissible defining function.
  Given $U \in \CHspace$ satisfying $L_{2k} U = 0$, there holds
  \begin{align}
    B^{2\gamma}_{2\gamma-2j}(U)=&c_{\gamma,j}P_{\gamma-2j}B^{2\gamma}_{2j}(U),\;\;0\leq j\leq\lfloor\gamma/2\rfloor;\\
    B^{2\gamma}_{2\lfloor\gamma\rfloor-2j}(U)=&d_{\gamma,j}P_{\lfloor\gamma\rfloor-[\gamma]-2j}B^{2\gamma}_{2j+2[\gamma]}(U),\;\;
    0\leq j\leq\lfloor\gamma\rfloor-\lfloor\gamma/2\rfloor-1,
  \end{align}
  where
  \begin{align}
    c_{\gamma,j}=&2^{2j-\g} \frac{\Gamma(2j-\g)}{\Gamma(\g-2j)};\\
    d_{\gamma,j}=&2^{2j+[\g]-\lfloor \g \rfloor} \frac{\Gamma(2j+[\g]-\lfloor \g \rfloor)}{\Gamma(\lfloor \g \rfloor-2j-[\g])}.
  \end{align}
\end{theorem}
\noindent
We emphasize that the fractional order conformally covariant operators are being recovered through a Dirichlet-to-Neumann-type map applied to a solution of an underdetermined problem.
Moreover, recovering the fractional operators via an underdetermined problem is what allows us to achieve the sharp higher order trace inequalities without stipulating extra Neumann-type boundary conditions the functions.

The third result is that the associated Dirichlet form
\begin{align*}
  \mathcal{Q}_{2\gamma}(U,V):=&\int_{X}U L_{2k} V \rho^{1-2[\gamma]}dvol_{\rho^2 g_+} \\
  &- \sum_{j=0}^{\lfloor \g/2 \rfloor} \sigma_{j,\g}  \int_{M} B_{2j}^{2\g}U  B_{2\g-2j}^{2\g}V  dvol_{g} - \sum_{j=\lfloor \g/2 \rfloor + 1}^{\lfloor \g \rfloor} \sigma_{j,\g} \int_{M}   B_{2\g-2j}^{2\g}U  B_{2j}^{2\g}V  dvol_{g}
\end{align*}
is symmetric for $U,V\in \mathcal{C}^{2\gamma}\cap\dot{H}^{k,[\gamma]}(X)$, where
\begin{equation}
  \sigma_{j,\g} =
  \begin{cases}
    2\frac{\Gamma(\g-j+1)\Gamma(j+1)\Gamma(j+\lfloor \g \rfloor - \g + 1)\Gamma(\lfloor \g \rfloor - j + 1)}{\Gamma(\g-2j) \Gamma(2j-\g+1)} & j = 0,\ldots, \lfloor \g/2 \rfloor\\
   2\frac{\Gamma(\g-j+1)\Gamma(j+1)\Gamma(j+\lfloor \g \rfloor - \g + 1)\Gamma(\lfloor \g \rfloor - j + 1)}{\Gamma(\g-2j + 1) \Gamma(2j-\g)} & j = \lfloor \g /2 \rfloor + 1,\ldots, \lfloor \g \rfloor\\
  \end{cases}.
  \label{eq:sigma-j-g}
\end{equation}
Evidently, by Theorem \ref{thm:boundary-operator-conformal-covariance}, $\mathcal{Q}_{2\g}$ is conformally covariant in the sense that, for $\hat\rho = e^{\tau} \rho$, there holds
\[
  \hat{\mathcal{Q}}_{2\g}(U,V) = \mathcal{Q}_{2\g}(e^{(\frac{n}{2} - \g)\tau}U, e^{(\frac{n}{2} - \g)\tau}V).
\]
Since $\mcQ_{2\g}$ is symmetric it is variational and so we may speak of the minimal energy of the corresponding Dirichlet energy functional
\[
  \mcE_{2\g}(V) := \mcQ_{2\g}(V,V).
\]
(See also Proposition \ref{prop:sepctral-assumption-proposition} for a characterization of the boundedness of the energy $\mcE_{2\g}$.)
The symmetry of $\mcQ_{2\g}$ is recorded in the following theorem.

\begin{theorem}\label{thm:dirichlet-form-symmetry}
  Let $U,V\in \mathcal{C}^{2\gamma}\cap\dot{H}^{k,[\gamma]}(X)$.
  It holds that
  \begin{align*}
    \mathcal{Q}_{2\gamma}(U,V)=\mathcal{Q}_{2\gamma}(V,U).
  \end{align*}
\end{theorem}

Lastly, we obtain all higher order sharp trace inequalities which relate the energy $\mcE_{2\g}$ of a function $U$ on $\overline{X}$ to boundary integrals in terms of the fractional order conformally covariant operators applied to the boundary data of $U$.
\begin{theorem}\label{th1.5}
  Let  $\gamma\in (0,\frac{n}{2})\setminus\mathbb{N}$ and $k = \lfloor \g \rfloor + 1$.
  Let $(X^{n+1},M^{n},g_{+})$ be a Poincar\'e-Einstein manifold satisfying $\la_{1}(-\Delta_{+}) > \frac{n^{2}}{4} - (2\lfloor \g \rfloor-\g)^{2}$.
  Let $\rho$ be a $\g$-admissible defining function and let $\mcE_{2\g}$ be the corresponding Dirichlet energy functional.
  Then, for $U\in \mathcal{C}^{2\gamma}\cap\dot{H}^{k,[\gamma]}(X)$, there holds
  \begin{align*}
    \mathcal{E}_{2\g}(U) \geq \sum_{j=0}^{\lfloor \g/2 \rfloor} \varsigma_{j,\g} \int_{M}   B_{2j}^{2\g}U P_{\g-2j}B_{2j}^{2\g} U    dvol_{g} + \sum_{j=\lfloor \g /2 \rfloor + 1}^{\lfloor \g \rfloor} \varsigma_{j,\g} \int_{M}   B_{2\g - 2j}^{2\g}U P_{2j - \g}B_{2\g-2j}^{2\g} U    dvol_{g}.
  \end{align*}
  where
  \begin{align*}
    \varsigma_{j,\g} &=
    \begin{cases}
      2^{2j-\g+1} \frac{\Gamma(\g-j+1)\Gamma(j+1)\Gamma(j+\lfloor \g \rfloor - \g + 1)\Gamma(\lfloor \g \rfloor - j + 1)}{\Gamma(\g-2j)\Gamma(\g-2j+1)} & j = 0,\ldots, \lfloor \g /2 \rfloor\\
      2^{\g-2j+1} \frac{\Gamma(\g-j+1)\Gamma(j+1)\Gamma(j+\lfloor \g \rfloor - \g +1)\Gamma(\lfloor \g \rfloor - j + 1)}{\Gamma(2j-\g)\Gamma(2j-\g+1)} & j = \lfloor \g/2 \rfloor + 1, \ldots, \lfloor \g \rfloor
    \end{cases}
  \end{align*}
  are positive constants.
  Moreover, equality is attained iff $L_{2k} U = 0$.
\end{theorem}

\begin{remark}\label{rem:energy-minimization}
  We may understand those $U$ which attain the equality
  \begin{align*}
    \mathcal{E}_{2\g}(U) =  \sum_{j=0}^{\lfloor \g/2 \rfloor} \varsigma_{j,\g} \int_{M}   B_{2j}^{2\g}U P_{\g-2j}B_{2j}^{2\g} U    dvol_{g} + \sum_{j=\lfloor \g /2 \rfloor + 1}^{\lfloor \g \rfloor} \varsigma_{j,\g} \int_{M}   B_{2\g - 2j}^{2\g}U P_{2j - \g}B_{2\g-2j}^{2\g} U    dvol_{g}.
  \end{align*}
  as being energy minimizers of $\mcE_{2\g}$.
  To see this, we first fix boundary data
  \begin{align*}
    f^{(2j)} \in C^{\oo}(M) \cap H^{\g-2j}(M) \qquad   &\text{ for } \qquad j = 0, \ldots \lfloor \g/2 \rfloor\\
    \phi^{(2j)}  \in C^{\oo}(M) \cap H^{\lfloor \g \rfloor - [\g] -2j}(M) \qquad & \text{ for } \qquad j = \lfloor \g/2 \rfloor + 1, \ldots, \lfloor \g \rfloor,
  \end{align*}
  and consider all of those $V\in \mathcal{C}^{2\gamma}\cap\dot{H}^{k,[\gamma]}(X)$ which satisfy
  \begin{align*}
    B_{2j}^{2\g} V  & = f^{(2j)},  \quad j = 0, \ldots \lfloor \g/2 \rfloor\\
    B_{2\g-2j}^{2\g} V &= \phi^{(2j)},  \quad j = \lfloor \g/2 \rfloor + 1, \ldots, \lfloor \g \rfloor.
  \end{align*}
  Then $\mcE_{2\g}$ is minimized exactly at the unique solution to the Dirichlet problem
  \begin{equation*}
    \begin{cases}
      L_{2k} V = 0, & \text{in }X\\
      B_{2j}^{2\g}V = f^{(2j)},& 0 \leq j \leq \lfloor \g/2 \rfloor\\
      B_{2j+2[\g]}^{2\g}V = \phi^{(2j)}, & 0 \leq j \leq \lfloor \g \rfloor - \lfloor \g/2 \rfloor -1
    \end{cases}.
  \end{equation*}
  See Proposition \ref{prop:sepctral-assumption-proposition} for classification of when the energy $\mcE_{2\g}$ may be minimized.
\end{remark}

\begin{remark}
  Note that, by \cite{MR1362652}, the spectral assumption on $(X,M,g_{+})$ holds whenever the conformal infinity has nonnegative Yamabe constant.
  In particular, it holds for $\B^{n+1}$ equipped with the standard Poincar\'e metric.
\end{remark}

\subsection{Application to the Real Hyperbolic Space}

As an important application of our results, we obtain the sharp higher order Sobolev trace inequalities on the ball $\B^{n+1} \subset \R^{n+1}$ and, using conformal covariance, we recover the analogous results on the halfspace $\R_{+}^{n+1}$.
Recall that the Poincar\'e ball model of the real hyperbolic space is the unit ball $\B^{n+1} \subset \R^{n+1}$ equipped with the Poincar\'e metric
\[
  g_{\B} = \frac{4|dw|^{2}}{(1-|w|^{2})^{2}}.
\]
The Laplace-Beltrami operator for $g_{\B}$ is
\[
  \Delta_{\B} = \frac{1-|w|^{2}}{4} \left\{ (1-|w|^{2})\Delta_{\R^{n+1}} + 2(n-1) \sum_{i=1}^{n+1} w_{i} \p_{w_{i}} \right\}, \quad \Delta_{\R^{n+1}} = \sum_{j=1}^{n+1} \p_{w_{j}}^{2}.
\]
For the standard round metric on the boundary $S^{n} = \p \B^{n+1}$, the geodesic defining function is
\[
  r = \frac{1-|w|^{2}}{2}.
\]
We will let $B_{\a}^{2\g,\B}$ denote the boundary operators for $\B^{n+1}$; their explicit expressions are given in Section \ref{sec:boundary-operators}.
Let
\begin{align*}
  D_{s}^{\B} = -\Delta_{\B} - s(n-s), \quad L_{2k}^{+,\B} = \prod_{j=0}^{k-1}D_{s-j}^{\B}, \quad s = \frac{n}{2} + \g\\
  L_{2k}^{\B} = \left( \frac{1-|w|^{2}}{2} \right)^{-\frac{n}{2}+\g-2k} \circ L_{2k}^{+,\B} \circ \left( \frac{1-|w|^{2}}{2} \right)^{\frac{n}{2}- \g}
\end{align*}
and define the corresponding Dirichlet form and energy
\begin{align*}
  \mathcal{Q}_{2\gamma}^{\B}(U,V):=&\int_{\B^{n+1}}U L_{2k}^{\B} V \left( \frac{1-|w|^{2}}{2} \right)^{1-2[\g]} dw\\
  &- \sum_{j=0}^{\lfloor \g/2 \rfloor} \sigma_{j,\g}  \int_{S^{n}} B_{2j}^{2\g,\B}U  B_{2\g-2j}^{2\g,\B}V d\sigma- \sum_{j=\lfloor \g/2 \rfloor + 1}^{\lfloor \g \rfloor} \sigma_{j,\g} \int_{S^{n}}   B_{2\g-2j}^{2\g,\B}U  B_{2j}^{2\g,\B}V d\sigma\\
  \mcE_{2\g}^{\B}(U) &= \mcQ_{2\g}^{\B}(U,U).
\end{align*}
Here, $d\sigma$ denotes the non-normalized round measure on $S^{n}$ and $\sigma_{j,\g}$ are given in \eqref{eq:sigma-j-g}.

Recall next that the halfspace model of the real hyperbolic space is
\[
  \R_{+}^{n+1} = \left\{ (x_{1},\ldots,x_{n},y) \in \R^{n} \times \R: y>0 \right\}
\]
equipped with the metric
\[
  g_{\H} = \frac{|dx|^{2} + dy^{2}}{y^{2}}.
\]
The Laplace-Beltrami operator for $g_{\H}$ is
\[
  \Delta_{\H} = y^{2}(\Delta_{x} +\p_{yy}) + (1-n) y \p_{y}, \quad \Delta_{x} = \sum_{j=1}^{n} \p_{x_{j}}^{2}.
\]
Note that $(x,y) \mapsto y$ is a geodesic defining function for $\R_{+}^{n+1}$ which corresponds to the flat metric on the boundary $\R^{n}$.
We will let $B_{\a}^{2\g,\H}$ denote the boundary operators for $\R_{+}^{n+1}$; their explicit expressions are also given in Section \ref{sec:boundary-operators}.
Let
\begin{align*}
  D_{s}^{\H} = -\Delta_{\H} &- s(n-s), \quad L_{2k}^{+,\H} = \prod_{j=0}^{k-1}D_{s-j}^{\H}, \quad s = \frac{n}{2} + \g\\
  L_{2k}^{\H} &= y^{-\frac{n}{2}+\g-2k} \circ L_{2k}^{+,\H} \circ y^{\frac{n}{2}- \g}
\end{align*}
and define the corresponding Dirichlet form and energy
\begin{align*}
  \mathcal{Q}_{2\gamma}^{\H}(U,V):=&\int_{\R_{+}^{n+1}}U L_{2k}^{\H} V y^{1-2[\g]} dxdy\\
  &- \sum_{j=0}^{\lfloor \g/2 \rfloor} \sigma_{j,\g}  \int_{\R^{n}} B_{2j}^{2\g,\H}U  B_{2\g-2j}^{2\g,\H}V dx - \sum_{j=\lfloor \g/2 \rfloor + 1}^{\lfloor \g \rfloor} \sigma_{j,\g} \int_{\R^{n}}   B_{2\g-2j}^{2\g,\H}U  B_{2j}^{2\g,\H}V dx\\
  \mcE_{2\g}^{\H}(U) &= \mcQ_{2\g}^{\H}(U,U),
\end{align*}
where the $\sigma_{j,\g}$ are given in \eqref{eq:sigma-j-g}.

Combining our trace inequalities in Theorem \ref{th1.5} with Beckner's sharp fractional Sobolev inequalities on the sphere obtained in \cite{MR1230930}, we obtain the higher order Sobolev trace inequalities on $\B^{n+1}$.
Using conformal covariance, we recover the analogous results on $\R_{+}^{n+1}$.
We first recall the fractional Sobolev inequalities of Beckner.
\begin{theoremalph}
  \label{thm:beckner-sobolev-inequality}
  Let $\Delta_{S^{n}}$ be the Laplace-Beltrami operator on the standard round sphere $(S^{n},g_{S^{n}})$ and define
  \begin{align*}
    P_{\g} = \frac{\Gamma(B + \frac{1}{2} + \g)}{\Gamma(B + \frac{1}{2} - \g)}, \quad B = \sqrt{-\Delta_{S^{n}} + \frac{(n-1)^{2}}{4}}.
  \end{align*}
  Then, for $0 < \g < \frac{n}{2}$, there holds
  \begin{equation}
    \frac{\Gamma(\frac{n+2\g}{2})}{\Gamma(\frac{n-2\g}{2})} \omega_{n}^{\frac{2\g}{n}} \left( \int_{S^{n}} |f|^{\frac{2n}{n-2\g}}d\sigma \right)^{\frac{n-2\g}{n}} \leq \int_{S^{n}} f P_{\g} f d\sigma.
    \label{eq:beckner-sobolev-inequality}
  \end{equation}
  Equality holds only for functions of the form
  \[
    c |1-\gen{x,\xi}|^{\frac{2\g-n}{2}}, \quad c \in \R, x \in \B^{n+1}, \xi \in S^{n}.	
  \]
  If $\g = \frac{n}{2}$, then
  \[
    \ln \left( \frac{1}{\omega_{n}} \int_{S^{n}} e^{f-\overline{f}} d\sigma \right) \leq \frac{1}{2n!\omega_{n}} \int_{S^{n}} f P_{\frac{n}{2}} f d\sigma.
  \]
  Equality holds only for functions of the form
  \[
    -n \ln |1-\gen{\zeta,\xi}| + c, \quad \zeta \in \B^{n+1}, \xi \in S^{n}, c \in \R.
  \]
\end{theoremalph}

The sharp higher order Sobolev inequalities on $\B^{n+1}$ are now stated.

\begin{theorem}
  Let  $\gamma\in (0,\frac{n}{2})\setminus\mathbb{N}$ and $k = \lfloor \g \rfloor + 1$.
  Then, for $U\in \mathcal{C}^{2\gamma}\cap\dot{H}^{k,[\gamma]}(\B^{n+1})$, there holds
  \begin{align*}
    \mcE_{2\g}^{\B}(U) &\geq \sum_{j=0}^{\lfloor \g /2 \rfloor} \varsigma_{j,\g} \omega_{n}^{\frac{2\g}{n}} \frac{\Gamma\left( \frac{n+2\g-4j}{2} \right)}{\Gamma\left( \frac{n-2\g+4j}{2} \right)} \left(\int_{S^{n}} |B_{2j}^{2\g,\B}U|^{\frac{2n}{n-2\g+2j}} d\sigma \right)^{\frac{n-2\g+4j}{n}} \\
    &+ \sum_{j=\lfloor \g/2 \rfloor + 1}^{\lfloor \g \rfloor} \varsigma_{j,\g} \omega_{n}^{\frac{2\g}{n}} \frac{\Gamma\left( \frac{n+4j-2\g}{2} \right)}{\Gamma\left( \frac{n-4j+2\g}{2} \right)} \left( \int_{S^{n}} |B_{2\g-2j}^{2\g,\B}U|^{\frac{2n}{n+2\g-4j}} d\sigma \right)^{\frac{n+2\g-4j}{n}}
  \end{align*}
  where the constants $\varsigma_{j,\g}$ are given in Theorem \ref{th1.5} and $\omega_{n}$ is the surface area of $S^{n}$.
  Moreover, equality is attained if and only if $U$ is the unique solution to
  \[
    \begin{cases}
      L_{2k}^{\B} U = 0 & \text{ in }\B^{n+1}\\
      B_{2j}^{2\g,\B} U = c_{2j} (1 - \gen{x,\xi_{2j}})^{-\frac{n-2\g+4j}{2}} & 0 \leq j \leq \lfloor \g/2 \rfloor \\
      B_{2\g-2j}^{2\g,\B} U = c_{2\g-2j} (1 - \gen{x,\xi_{2\g-2j}})^{-\frac{n-4j+2\g}{2}} & \lfloor \g/2 \rfloor + 1 \leq j \leq \lfloor \g \rfloor
    \end{cases},
  \]
  for some constants $c_{\a} \in \R$ and points $\xi_{\a} \in S^{n}$.
\end{theorem}

We lastly demonstrate that our results also recover the higher order extension theorems and Sobolev trace inequalities on the real halfspace $\R_{+}^{n+1}$.
(These results should be compared with those of Case in \cite{MR4095805}.)
To see this, recall that the M\"obius transform $\mathcal{M}:\R_{+}^{n+1} \to \B^{n+1}$, defined by
\[
  \mathcal{M}(x,y) = \left( \frac{2x}{(1+y)^{2} + |x|^{2}}, \frac{1-|x|^{2} - y^{2}}{(1+y)^{2} +|x|^{2}} \right),
\]
is an isometry between $(\R_{+}^{n+1},g_{\H})$ and $(\B^{n+1},g_{\B})$.
Note $(\Delta_{\B}v) \circ \mathcal{M} = \Delta_{\H}(v \circ \mathcal{M})$ for $v  \in C^{\oo}(\B^{n+1})$.
Moreover, letting $w = \mcM(x,y)$ for $(x,y) \in \R_{+}^{n+1}$, there holds
\begin{align*}
  \frac{1-|w|^{2}}{2} &= \frac{2y}{(1+y)^{2} + |x|^{2}}\\
  J_{\mcM} &= \left( \frac{2}{(1+y)^{2} + |x|^{2}} \right)^{n+1},
\end{align*}
where $J_{\mcM}$ is the Jacobian of $\mcM$.
Restricting $\mathcal{M}$ to the boundary $S^{n}$ defines a map $\mathcal{C} = \p\mathcal{M}$, which is the usual stereographic projection.
Importantly, if $P_{\g}^{S^{n}}$ are the fractional GJMS operators on $S^{n}$, then
\[
  (P_{\g}^{S^{n}}F) \circ \mathcal{C} = J_{\mathcal{C}}^{-\frac{n+2\g}{2n}} (-\Delta_{x})^{\g} \left( J_{\mathcal{C}}^{\frac{n-2\g}{2n}}(F \circ \mathcal{C}) \right), \quad  F \in C^{\oo}(S^{n})
\]
where
\[
  J_{\mathcal{C}} = \left( \frac{2}{1+|x|^{2}} \right)^{n}
\]
is the Jacobian of $\mathcal{C}$.
These maps allow us to translate the results on $\B^{n+1}$ to results on $\R^{n+1}_{+}$.

Using the Cayley transform, we may relate the boundary operators $B_{\a}^{2\g,\B}$ with $B_{\a}^{2\g,\H}$ as established in the following theorem.

\begin{theorem}
  Given $U \in \mcC^{2\g}(\B^{n+1})$, then for all $0 \leq j \leq \lfloor \g \rfloor$ there holds
  \begin{align*}
    B_{2j}^{2\g,\B}U &= J_{\mcC}^{\frac{-n+2\g-4j}{2n}} B_{2j}^{2\g,\H}\left( J_{\mcM}^{\frac{n-2\g}{2n+2}}U \circ \mcM \right)\\
    B_{2j+2[\g]}^{2\g,\B}U &= J_{\mcC}^{\frac{-n+2\g-4j-4[\g]}{2n}} B_{2j+2[\g]}^{2\g,\H}\left( J_{\mcM}^{\frac{n-2\g}{2n+2}} U \circ \mcM \right).
  \end{align*}
  \label{thm:boundary-on-ball-to-on-halfspace}
\end{theorem}
Evidently, this implies that
\[
  \mcQ_{2\g}^{\H}(U,V) = \mcQ_{2\g}^{\B}(J_{\mcM}^{\frac{-n+2\g}{2n+2}}U\circ \mcM^{-1}, J_{\mcM}^{\frac{-n+2\g}{2n+2}}V \circ \mcM^{-1}).
\]
Using Theorem \ref{thm:boundary-on-ball-to-on-halfspace} with Theorem \ref{thm:boundary-to-fractional-operator} and the conformal covariance between $(-\Delta_{x})^{\g}$ and $P_{\g}^{S^{n}}$, we may recover the fractional powers $(-\Delta_{x})^{\g}$ on $\R^{n}$ via our boundary operators.
This is recorded in the following theorem.
\begin{theorem}
  \label{thm:recovering-fractional-laplacian-on-Rn}
  Let  $\gamma\in (0,\frac{n}{2})\setminus\mathbb{N}$ and $k = \lfloor \g \rfloor + 1$.
  Given $U \in \mcC^{2\g}(\R_{+}^{n+1}) \cap \dot{H}^{k,\g}(\R_{+}^{n+1})$ satisfying $L_{2k}^{\H} U = 0$, there holds
  \begin{align*}
    B^{2\gamma,\H}_{2\gamma-2j}(U)=&c_{\gamma,j}(-\Delta_{x})^{\g-2j}B^{2\gamma,\H}_{2j}(U),\;\;0\leq j\leq\lfloor\gamma/2\rfloor;\\
    B^{2\gamma,\H}_{2\lfloor\gamma\rfloor-2j}(U)=&d_{\gamma,j}(-\Delta_{x})^{\lfloor \g \rfloor - [\g] -2j} B^{2\gamma,\H}_{2j+2[\gamma]}(U),\;\;
    0\leq j\leq\lfloor\gamma\rfloor-\lfloor\gamma/2\rfloor-1,
  \end{align*}
  where
  \begin{align*}
    c_{\gamma,j}=&2^{2j-\g} \frac{\Gamma(2j-\g)}{\Gamma(\g-2j)};\\
    d_{\gamma,j}=&2^{2j+[\g]-\lfloor \g \rfloor} \frac{\Gamma(2j+[\g]-\lfloor \g \rfloor)}{\Gamma(\lfloor \g \rfloor-2j-[\g])}.
  \end{align*}
\end{theorem}

Putting this altogether, we obtain the following sharp higher order Sobolev trace inequalities on $\R_{+}^{n+1}$.
\begin{theorem}
  Let  $\gamma\in (0,\frac{n}{2})\setminus\mathbb{N}$ and $k = \lfloor \g \rfloor + 1$.
  Then, for $U\in \mathcal{C}^{2\gamma}\cap\dot{H}^{k,[\gamma]}(\R_{+}^{n+1})$, there holds
  \begin{align*}
    \mcE_{2\g}^{\H}(U) &\geq \sum_{j=0}^{\lfloor \g /2 \rfloor} \varsigma_{j,\g} \omega_{n}^{\frac{2\g}{n}} \frac{\Gamma\left( \frac{n+2\g-4j}{2} \right)}{\Gamma\left( \frac{n-2\g+4j}{2} \right)} \left(\int_{\R^{n}} |B_{2j}^{2\g,\H}U|^{\frac{2n}{n-2\g+2j}} dx \right)^{\frac{n-2\g+4j}{n}} \\
    &+ \sum_{j=\lfloor \g/2 \rfloor + 1}^{\lfloor \g \rfloor} \varsigma_{j,\g} \omega_{n}^{\frac{2\g}{n}} \frac{\Gamma\left( \frac{n+4j-2\g}{2} \right)}{\Gamma\left( \frac{n-4j+2\g}{2} \right)} \left( \int_{\R^{n}} |B_{2\g-2j}^{2\g,\H}U|^{\frac{2n}{n+2\g-4j}} dx \right)^{\frac{n+2\g-4j}{n}}
  \end{align*}
  where the constants $\varsigma_{j,\g}$ are given in Theorem \ref{th1.5}.
  Moreover, equality is attained if and only if $U$ is the unique solution to
  \[
    \begin{cases}
      L_{2k}^{\H} U = 0 & \text{ in }\R_{+}^{n+1}\\
      B_{2j}^{2\g,\H} U = a_{2j} (\e_{2j} + |x - \xi_{2j}|^{2})^{-\frac{n-2\g+4j}{2}} & 0 \leq j \leq \lfloor \g/2 \rfloor \\
      B_{2\g-2j}^{2\g,\H} U = a_{2\g-2j} (\e_{2\g-2j} + |x - \xi_{2\g-2j}|^{2})^{-\frac{n-4j+2\g}{2}} & \lfloor \g/2 \rfloor + 1 \leq j \leq \lfloor \g \rfloor
    \end{cases},
  \]
  for some constants $a_{\a} \in \R$, $\e_{\a} \in (0,\oo)$ and points $\xi_{\a} \in \R^{n}$.
\end{theorem}

\subsection{Outline of Paper}

We conclude the introduction with an outline of the paper.
In the next section, we provide precise definitions for Poincar\'e-Einstein manifolds, $\g$-admissible defining functions and define the function spaces such as $\mcC^{2\g}$ and $\ceven$.
In Section \ref{sec:boundary-operators}, we introduce our boundary operators and state and prove important properties used in later sections.
In Section \ref{sec:dirichlet-problem-and-scattering}, we study the Dirichlet problem related to the weighted operators $L_{2k}$.
This will be needed for establishing the extension and trace inequality results and relies heavily on the Graham-Zworski scattering theory.
In Section \ref{sec:proofs-section}, we provide the proofs of our main results.
Of particular importance is our main integral (Theorem \ref{thm:main-integral-identity}) identity used to prove the symmetry of $\mcQ_{2\g}$ and the Sobolev trace inequalities.
The majority of the computational work is in proving this integral identity.
In Section \ref{sec:spectral-assumption} we prove a result concerning characterizing the boundedness of $\mcE_{2\g}$.

\section{Preliminaries}
\label{sec:preliminaries}
Let $n \geq 3$.
Throughout, $\overline{X^{n+1}}$ denotes a smooth compact $(n+1)$-dimensional manifold with $n$-dimensional boundary $M^{n}$ and interior $X^{n+1}$.
As is standard, we will often write $\overline{X},X,M$ for $\overline{X^{n+1}},X^{n+1},M^{n}$.
We call $(X,M,g_{+})$ a Poincar\'e-Einstein manifold provided $g_{+}$ is a complete Einstein metric on $X$ with $\Ric(g_{+}) = -ng_{+}$ and there exists a smooth defining function $\rho:\overline X \to \R_{\geq0}$ with $\rho^{-1}(0) = M$, $|d\rho|_{g}^{2} = 1$ and such that $g=\rho^{2}g_{+}$ extends to a $C^{n-1,\a}$-metric on $\overline X$.
Note that the existence of a defining function defines a conformal structure on $M$ since, if $\rho$ is a defining function, so too is $e^{\tau}\rho$ for any $\tau \in C^{\oo}(\overline X)$.
Let $[g]$ denote the conformal class on $M$.
To each $h \in [g]$ is a unique defining function $r$, called a geodesic defining function, on $M \times (0,\e)$ for some $\e>0$ such that
\begin{equation}
  g_{+} = \frac{dr^{2} + h_{r}}{r^{2}} \quad \text{ on } M \times (0,\e)
  \label{eq:metric-in-normal-form}
\end{equation}
and
\begin{align*}
  h_{r} &= h + h_{(2)} r^{2} + \cdots + h_{(n-1)}r^{n-1} + k r^{n} + o(r^{n}) \quad \text{ for } n \text{ odd}\\
  h_{r} &= h + h_{(2)} r^{2} + \cdots + h_{(n-2)}r^{n-2} + h_{(n)}r^{n} \log r +  k r^{n} + o(r^{n}) \quad \text{ for } n \text{ even},
\end{align*}
where $h_{(r)}$ are locally determined by $h$ and $k$ is nonlocal.
Here and elsewhere, we often identify a collar neighborhood of $M$ with $M \times (0,\e)$.

Fix a geodesic defining function $r$ (and hence a metric on $M$) and let $\g \in (0,\frac{n}{2}) \setminus \N$.
A defining function $\rho$ of $M$ is called $\g$-admissible provided
\begin{equation}
  \frac{\rho}{r} = 1 + \sum_{j=1}^{\lfloor \g \rfloor} \rho_{(2j)} r^{2j} + \Phi r^{2 \g} + o(r^{2\g}) \text{ as } r \to0,
  \label{eq:gamma-admissible-defining-function}
\end{equation}
where $\rho_{(2j)},\Phi \in C^{\oo}(M)$.
Clearly, if $\rho$ is $\g$-admissible, then
\begin{equation}
  \frac{r}{\rho} = 1 + \sum_{j=1}^{\lfloor \g \rfloor} \hat\rho_{(2j)} \rho^{2j} + \hat\Phi \rho^{2 \g} + o(\rho^{2\g}) \text{ as } \rho \to0,
\end{equation}
for some $ \hat\rho_{(2j)},\hat\Phi \in C^{\oo}(M)$.
Note that, if $\rho$ is $\g$-admissible, then $\rho$ is conformal to $r$ via $\rho = e^{\tau} r$, where $\tau = \log \rho/r$, which is smooth by \eqref{eq:gamma-admissible-defining-function}.
Fix a $\g$-admissible $\rho$.
Now, let
\[
  \mathcal{C}^{2\g} = C_{\operatorname{even}}^{\oo}  + \rho^{2 [\g]} C_{\operatorname{even}}^{\oo},
\]
where $f \in C_{\operatorname{even}}^{\oo}$ indicates that $f$ is in $C^{\oo}(\overline{X})$ and has a Taylor expansion in terms of even powers of $\rho$.
To provide the characterization of the boundedness of the energy $\mcE_{2\g}$, it will be convenient to introduce the following notation.
(This notation is only used in Section \ref{sec:spectral-assumption} and is based on that given in \cite{MR3619870} for $\g \in(0,1) \cup (1,2)$.)
For $\lfloor \g \rfloor \in 2 \N$, we let $\vec f$ and $\vec \psi$ denote vectors in $C^{\oo}(M)^{\frac{\lfloor \g \rfloor}{2}+1}$ and $C^{\oo}(M)^{\frac{\lfloor \g \rfloor}{2}}$, respectively (when $\lfloor \g \rfloor = 0$, we do not consider $\vec \psi$).
For $\lfloor \g \rfloor \in 2\N + 1$, we let $\vec f$ and $\vec \psi$ denote vectors in $C^{\oo}(M)^{\frac{\lfloor \g \rfloor+1}{2}}$.
Whenever $\vec f$ or $\vec \psi$ is used, it will always mean such vectors of smooth functions and appropriate dimension.
Given $\g \in (0,\frac{n}{2}) \setminus \N$ and $\vec f, \vec \psi$, we let $\mcC_{\vec f, \vec \psi}^{2\g}$ denote the space of functions $U \in \mcC^{2\g}$ that may be asymptotically expanded near $M$ as
\begin{align*}
  U &= f_{0} + f_{2} \rho^{2} + \cdots + f_{\lfloor \g \rfloor } \rho^{\lfloor \g \rfloor}\\
  &+ \psi_{0}\rho^{2[\g]} + \psi_{2}\rho^{2 + 2[\g]} + \cdots + \psi_{\lfloor \g \rfloor-2} \rho^{\lfloor \g \rfloor + 2[\g]-2} + \psi_{*}\rho^{\lfloor \g \rfloor + 2[\g] } + o(\rho^{\lfloor \g \rfloor + 2[\g] })\\
  \vec f &= (f_{0},f_{2},\ldots, f_{\lfloor \g \rfloor})\\
  \vec \psi &= (\psi_{0},\psi_{2},\ldots, \psi_{\lfloor \g \rfloor - 2})
\end{align*}
for some $\psi_{*} \in C^{\oo}(M)$, when $\lfloor \g \rfloor \in 2 \N$, and as
\begin{align*}
  U &= f_{0} + f_{2} \rho^{2} + \cdots +  f_{\lfloor \g \rfloor - 1}\rho^{\lfloor \g \rfloor - 1} + f_{*}\rho^{\lfloor \g \rfloor + 1} \\
  &+ \psi_{0}\rho^{2[\g]} + \psi_{2}\rho^{2 + 2[\g]} + \cdots \psi_{\lfloor \g \rfloor-1} \rho^{\lfloor \g \rfloor + 2[\g]-1} + \psi_{*}\rho^{\lfloor \g \rfloor + 2[\g] + 1} + o(\rho^{\lfloor \g \rfloor + 2[\g] + 1}),
\end{align*}
for some $f_{*},\psi_{*} \in C^{\oo}(M)$, when $\lfloor \g \rfloor \in 2 \N +1$.
Note
\begin{align}
  \mcC^{2\g} &= \bigcup_{\vec f, \vec \psi} \mcC_{\vec f,\vec \psi}^{2\g}\label{eq:C-2g-space}\\
  \ceven &= \bigcup_{\vec f} \mcC_{\vec f,0}^{2\g}\label{eq:D-2g-space}.
\end{align}
The $C^{2\g}$ are the appropriate spaces to establish the extension results and trace inequalities and are clearly independent of choice of $\g$-admissible defining function $\rho$.

Next, let $\dot H^{k,\g}(X)$ denote the closure of $C_{0}^{\oo}(\overline{X})$ with respect to $\mathcal{E}_{2\g}(U)^{1/2} = \mathcal{Q}_{2\g}(U,U)^{1/2}$.
As is standard, we define $H^{\g}(X)$ as the completion of $C^{\oo}(M)$ under the norm given via the pull back on coordinate charts of the fractional Sobolev norm on $\R^{n}$.
Whenever appropriate integrability is needed, we assume functions on $\overline{X}$ to belong to $\mcC^{2\g}(X) \cap \dot H^{k,\g}(X)$ and boundary functions to belong to $H^{\g}(M)$.

Henceforth we assume that each geodesic defining function $r$ has been extended from a collar neighborhood of $M$ to all of $\overline{X}$.
Since we will only care about the behavior of defining functions near $M$, the way in which $r$ is extended into $\overline X$ does not matter.
In any case, any extension of $r$ into $\overline{X}$ is automatically $\g$-admissible.

\section{Boundary Operators}
\label{sec:boundary-operators}

In this section we define the boundary operators on general Poincar\'e-Einstein manifolds.
We then establish some of their properties.
Thus, let $(X,M,g_{+})$ be a Poincar\'e-Einstein manifold, let $g$ be a chosen conformal infinity and let $r$ be the corresponding geodesic defining function so that we may write $g_{+}$ near $M$ using \eqref{eq:metric-in-normal-form}.
Let $\d$ be the divergence operator relative to $h_{0} = g$.
Following C. Fefferman and Graham \cite{MR3073887}, we introduce the following notations:
\begin{align*}
  V(r) &= \sqrt{\frac{\det h_{r}}{\det h_{0}}}\\
  W(r) &= \sqrt{V(r)} = 1 + \sum_{N \geq 1} W_{2N}r^{2N}\\
  \mcM &= \d(h_{r}^{-1} d) - U(r)\\
  U(r) &= \frac{\left[ \p_{rr} - (n-1) r^{-1} \p_{r} + \d(h_{r}^{-1}d) \right]W(r)}{W(r)}.
\end{align*}
Note that $W_{2N}$ is defined only for $1 \leq N \leq n/2$ when $n$ is even.
Note also that $\mcM(r)$ is a family of differential operators on $M$ and $W(0) = 1$.
The Laplace-Beltrami operator on $(M\times(0,\e),g_{+})$ may be expressed as
\[
  \Delta_{+} = r^{2}\left[ \p_{rr} + \Delta_{h_{r}} + 2 \frac{\p_{r} W(r)}{W(r)}\p_{r} + (1-n)\frac{1}{r} \p_{r} \right].
\]
Importantly, $\Delta_{+}$ satisfies the following intertwining formula
\begin{equation}
  W(r) \circ \Delta_{+} \circ W(r)^{-1} = r^{2}\left[ \p_{rr} + \mcM(r) \right] + (1-n)r\p_{r}.
  \label{eq:intertwining-formula}
\end{equation}
It follows that, if $u$ solves the Poisson equation
\[
  \begin{cases}
    -\Delta_{+} u - s(n-s) u = 0 & \text{ in } M \times (0,\e)\\
    u = r^{n-2} F + r^{s} H\\
    F|_{M} = f
  \end{cases}
\]
then $v = W u$ solves
\[
  \begin{cases}
    -\left[ r^{2} \p_{rr} + r^{2} \mcM(r) + (1-n)r\p_{r} \right] v - s(n-s) v = 0 & \text{ in } M \times (0,\e)\\
    v = r^{n-s} W(r) F + r^{s} W(r) H\\
    F|_{M} = f
  \end{cases}.
\]

We now introduce the family of boundary operators $B_{\a}^{2\g}$.
We will first define the operators and second justify that they are well-defined and natural.
To begin, we first introduce the notation
\begin{align*}
  \tilde\Delta_{+} &= \Delta_{+} + \frac{n^{2}}{4},
\end{align*}
let $\g \in (0,\frac{n}{2})\setminus\N$, let $\lfloor \g\rfloor$ be the integer part of $\g$, let $[\g] = \g - \lfloor \g \rfloor$ be the fractional part and let $k = \lfloor \g \rfloor + 1$.
Fix a $\g$-admissible defining function $\rho$ and let $U \in \mathcal{C}^{2\g}$.
For any integer $j\geq0$, we define the preliminary operators
\begin{align*}
  \tilde B_{2j}^{2\g } &= \rho^{-\frac{n}{2} + \g - 2j}  \circ \prod_{\ell=0}^{j-1}\left( \tilde\Delta_{+} - (\g -2\ell)^{2} \right) \left( \tilde \Delta_{+} - (\g + 2\ell - 2\lfloor \g \rfloor )^{2} \right) \circ  \rho^{\frac{n}{2}-\g } |_{\rho=0};\\
  \tilde B_{2j + 2[\g]}^{2\g}
  &= \rho^{-\frac{n}{2} + \g - 2j  - 2[\g]}\circ \prod_{\ell=0}^{j} \left( \tilde\Delta_{+} - (\g -2\ell)^{2} \right) \prod_{\ell=0}^{j - 1} \left( \tilde \Delta_{+} - (\g  + 2\ell - 2\lfloor \g \rfloor )^{2} \right) \circ \rho^{\frac{n}{2}-\g}  |_{\rho=0}
\end{align*}
and constants
\begin{align*}
  b_{2j} &= \tilde B_{2j}^{2\g } (\rho^{2j})=j!\frac{\Gamma(j+1-[\gamma])}{\Gamma(1-[\gamma])}
  \frac{\Gamma(\gamma+1-j)}{\Gamma(\gamma+1-2j)}
  \frac{\Gamma(\lfloor \g \rfloor+1-j)}{\Gamma(\lfloor \g \rfloor+1-2j)}\\
  b_{2j+2[\g]} &=  \tilde B_{2j+2[\g]}^{2\g } (\rho^{2j+2[\g]})=-j!\frac{\Gamma(j+1+[\gamma])}{\Gamma([\gamma])}
  \frac{\Gamma(\lfloor\gamma\rfloor+1-j)}{\Gamma(\lfloor\gamma\rfloor-2j)}
  \frac{\Gamma(\lfloor \g \rfloor+1-j-[\gamma])}{\Gamma(\lfloor \g \rfloor+1-2j-[\gamma])}.
\end{align*}
(Refer to Lemma \ref{lm1.2} to see how the constants are computed.)
The goal is for the boundary operators to pick out coefficients in the expansion of $U$ in terms of $\rho$; however, one readily finds
\begin{align*}
  b_{2j} &= 0 \quad \text{ for } \quad  \frac{1}{2} (1 + \lfloor \g \rfloor) \leq j \leq \lfloor \g \rfloor\\
  b_{2j + 2[\g]} &=0 \quad \text{ for } \quad \frac{1}{2}\lfloor \g \rfloor \leq j \leq \lfloor \g \rfloor.
\end{align*}
We may thus first define boundary operators only for the indicated index ranges:
\begin{align*}
  B_{2 j}^{2\g}U=&\frac{1}{ b_{ 2j }}\tilde B_{ 2j}^{2\g}U,\;\; \;\;\;\;\;\;\;\;\;\;\;\;\;\;\;\;\;\;0\leq j<\frac{1}{2}(1+\lfloor\gamma\rfloor); \\
  B_{ 2j + 2[\g]}^{2\g }U=& \frac{1}{ b_{ 2j + 2[\g]}} \tilde B_{ 2j + 2[\g]}^{2\g}U, \;\;\;\;\;\;\;\;0\leq j<\frac{1}{2}\lfloor\gamma\rfloor.
\end{align*}
Note that we are using the convention that empty products equal one and so
\[
  B_{0}^{2\g} U = U|_{\rho=0}.
\]
To account for the cases in which the constants vanish, we introduce the following parameterized operators:
\begin{align*}
  \tilde B_{2j}^{2\g,\epsilon } &= \rho^{-\frac{n}{2} + \g - 2j}  \circ{\left( \tilde\Delta_{+} - (\g -2j)^{2} +\epsilon\right)} \prod_{\ell=0}^{j-1}\left( \tilde\Delta_{+} - (\g -2\ell)^{2} \right)\\
  &
  \circ\prod_{\ell\in\{\lfloor\gamma\rfloor-j+1,\cdots,\lfloor\gamma\rfloor \}\setminus\{j\}}\left( \tilde \Delta_{+} - (\g -2\ell )^{2} \right) \circ  \rho^{\frac{n}{2}-\g } |_{\rho=0};\\
  \tilde B_{2j + 2[\g]}^{2\g,\epsilon} &=  \rho^{-\frac{n}{2} + \g - 2j  - 2[\g]}  \circ\left( \tilde\Delta_{+} - (\g -2\lfloor\gamma\rfloor+2j)^{2} +\epsilon\right) \prod_{\ell\in\{0,\cdots,j\}\setminus\{\lfloor \gamma\rfloor-j\}}\left( \tilde\Delta_{+} - (\g -2\ell)^{2} \right)\\
  &
  \circ \prod_{\ell=\lfloor\gamma\rfloor-j+1}^{\lfloor\gamma\rfloor}\left( \tilde \Delta_{+} -(\g -2\ell )^{2} \right) \circ  \rho^{\frac{n}{2}-\g } |_{\rho=0}
\end{align*}
and constants
\begin{align*}
  \tilde b_{2j}^{\e} &=  \tilde B_{2j}^{2\g, \e} (\rho^{2j})\\
  \tilde b_{2j + 2[\g]}^{\e} &=  \tilde B_{2j+ 2[\g]}^{2\g, \e} (\rho^{2j+2[\g]}).
\end{align*}
We may then define the following auxiliary operators
\begin{align*}
  \overline{B}_{2 j}^{2\g}U=&\lim_{\e \to 0}\frac{1}{\tilde b_{ 2j }^{\e}} \tilde B_{ 2j}^{2\g,\e}U, \quad  \frac{1}{2} (1 + \lfloor \g \rfloor) \leq j \leq \lfloor \g \rfloor\\
  \overline{B}_{2 j + 2[\g]}^{2\g}U&=\lim_{\e \to 0}\frac{1}{\tilde b_{ 2j + 2[\g] }^{\e}} \tilde B_{ 2j + 2[\g]}^{2\g,\e}U, \quad \frac{1}{2}\lfloor \g \rfloor \leq j \leq \lfloor \g \rfloor.
\end{align*}
While these operators no longer annihilate $\rho^{2j}$ and $\rho^{2j+2[\g]}$, respectively, they introduce another issue, namely, $\overline{B}_{2j}^{2\g} (f\rho^{2m+2[\g]})$ and $\overline{B}_{2j+2[\g]}^{2\g}(f\rho^{2m})$ may be infinite for small integers $m$ and $f \in C^{\oo}(\overline{X})$.
This issue is resolved by defining the boundary operators $B_{2j}^{2\g}$ and $B_{2j+2[\g]}^{2\g}$ as $\overline{B}_{2j}^{2\g}$ and $\overline{B}_{2j+2[\g]}^{2\g}$, respectively, acting only on the part of the expansion for $U$ for which they operators are finite.
To achieve this and to prove properties of these boundary operators, it will be convenient to introduce the following jets:
\begin{align*}
  J_{+}^{0}U &= B_{0}^{2\g}U \\
  J_{+}^{2m} U &= J_{+}^{0}U + \rho^{2} B_{2}^{2\g}(U - J_{+}^{0}U) \cdots + \rho^{2m} B_{2m}^{2\g}(U - J_{+}^{2m-2}U) \\
  J_{-}^{0}U &= \rho^{2[\g]} B_{2[\g]}^{2\g}U \\
  J_{-}^{2m} U & = J_{-}^{0} U + \rho^{2[\g]+2}B_{2+ 2[\g]}^{2\g}(U - J_{-}^{0}U) + \cdots +  \rho^{2m + 2[\g]}B_{2m+2[\g]}^{2\g}(U - J_{-}^{2m-2}U),
\end{align*}
where $J_{+}^{2m}$ is only defined for $0 \leq m < \frac{1}{2}( \lfloor \g \rfloor + 1)$ and $J_{-}^{2m}$ is only defined for $0 \leq m < \frac{1}{2}\lfloor \g \rfloor $.
Using Lemmas \ref{lem:example-computations-of-B} and \ref{lem:some-kernel-elements-of-B} below, one readily finds that $U - J_{-}^{2\lfloor \g \rfloor - 2j}U$ no longer has any problematic terms for $\overline{B}_{2j}^{2\g}$ and likewise for $U - J_{+}^{2\lfloor \g \rfloor - 2j}U$ and $B_{2j + 2[\g]}^{2\g}$.
At last, we therefore define the boundary operators
\begin{align*}
  B_{2j}^{2\g} U &= \overline{B}_{2j}^{2\g}(U - J_{-}^{2\lfloor \g \rfloor - 2j}U), \quad  \frac{1}{2} (1 + \lfloor \g \rfloor) \leq j \leq \lfloor \g \rfloor\\
  B_{2j+2[\g]}^{2\g} U &= \overline{B}_{2j+2[\g]}^{2\g} (U - J_{+}^{2\lfloor \g \rfloor - 2j}U), \quad \frac{1}{2}\lfloor \g \rfloor \leq j \leq \lfloor \g \rfloor.
\end{align*}

\begin{remark}In this paper, we try a new way of defining the conformally covariant boundary operators. We can also define the boundary operators which are conformally covariant using  the same method recently discovered in \cite{FLY23CRSobolev}:
\begin{itemize}
  \item $\lfloor\gamma/2\rfloor+1\leq j\leq\lfloor\gamma\rfloor$,
    \begin{align*}
B_{ 2j}^{2\g }(V)=&\frac{1}{b_{2j}}\rho^{1+2\gamma-4j}\partial_{\rho}\left[\rho^{2j-\frac{n}{2}-\gamma}
    \Pi^{\g}_{j}(\rho^{\frac{n}{2}-\g}(V-\tilde V))\right]\big|_{\rho=0}\\
& \varsigma_{j,\g}  \frac{1}{b_{2j}} P_{2j-\gamma}B^{2\gamma}_{2\gamma-2j}(V),
\end{align*}
  \item $\lfloor\gamma\rfloor-\lfloor\gamma/2\rfloor\leq j\leq\lfloor\gamma\rfloor$,
  \begin{align*}
B_{ 2j+2[\gamma]}^{2\g }(V)=&\frac{1}{b_{2j+2[\gamma]}}\rho^{1+2\gamma-4[\gamma]-4j}\partial_{\rho}\left[\rho^{2[\gamma]+2j-\frac{n}{2}-\gamma}
    \Pi^{\g}_{\lfloor\gamma\rfloor-j}(\rho^{\frac{n}{2}-\g}(V-\tilde V))\right]\big|_{\rho=0} \\
&+ \varsigma_{j,\g} \frac{1}{b_{2j+2[\gamma]}} P_{2j+[\gamma]-\lfloor\gamma\rfloor}B^{2\gamma}_{2\lfloor\gamma\rfloor-2j}(V);
\end{align*}
\end{itemize}
where $\Pi^{\g}_{j}$ and $\Pi^{\g}_{\lfloor\gamma\rfloor-j}$ are defined in  Section \ref{sec:proofs-section} and $\tilde{V}$ is the solution of \ref{eq:boundary-value-dirichlet-problem-l2k-copied}. Using these boundary operators,
we can also obtain the Sobolev trace inequalities on Poincar\'e-Einstein manifolds. We shall omit the details here.

\end{remark}

On the ball, we consider only the boundary operators acting on $U \in \mcC^{2\g}(\B^{n+1})$:
\begin{align*}
  B_{2 j}^{2\g,\B}U=&\frac{1}{ b_{ 2j }}\tilde B_{ 2j}^{2\g,\B}U,\;\; \;\;\;\;\;\;\;\;\;\;\;\;\;\;\;\;\;\;0\leq j<\frac{1}{2}(1+\lfloor\gamma\rfloor) \\
  \tilde B_{2j}^{2\g,\B } &= \left( \frac{1-|w|^{2}}{2} \right)^{-\frac{n}{2} + \g - 2j}  \\
  &\circ \prod_{\ell=0}^{j-1}\left( \tilde\Delta_{\B} - (\g -2\ell)^{2} \right) \left( \tilde \Delta_{\B} - (\g + 2\ell - 2\lfloor \g \rfloor )^{2} \right) \circ  \left( \frac{1-|w|^{2}}{2} \right)^{\frac{n}{2}-\g } \bigg|_{|w|=1}\\
  B_{ 2j + 2[\g]}^{2\g,\B }U=& \frac{1}{ b_{ 2j + 2[\g]}} \tilde B_{ 2j + 2[\g]}^{2\g,\B}U, \;\;\;\;\;\;\;\;0\leq j<\frac{1}{2}\lfloor\gamma\rfloor\\
  \tilde B_{2j + 2[\g]}^{2\g,\B}
  &= \left( \frac{1-|w|^{2}}{2} \right)^{-\frac{n}{2} + \g - 2j  - 2[\g]}\\
  &\circ \prod_{\ell=0}^{j} \left( \tilde\Delta_{+} - (\g -2\ell)^{2} \right) \prod_{\ell=0}^{j - 1} \left( \tilde \Delta_{+} - (\g  + 2\ell - 2\lfloor \g \rfloor )^{2} \right) \circ \left( \frac{1-|w|^{2}}{2} \right)^{\frac{n}{2}-\g}  \bigg|_{|w|=1}
\end{align*}
The definitions for $B_{2j}^{2\g,\B}$ and $B_{2j+2[g]}^{2\g,\B}$ for larger values of $j$ are defined analogously.
On the halfspace $\R_{+}^{n+1}$, we consider only the boundary operators acting on $U \in \mcC^{2\g}(\R_{+}^{n+1})$:
\begin{align*}
  B_{2 j}^{2\g,\H}U=&\frac{1}{ b_{ 2j }}\tilde B_{ 2j}^{2\g,\H}U,\;\; \;\;\;\;\;\;\;\;\;\;\;\;\;\;\;\;\;\;0\leq j<\frac{1}{2}(1+\lfloor\gamma\rfloor) \\
  \tilde B_{2j}^{2\g,\H } &= y^{-\frac{n}{2} + \g - 2j}\circ \prod_{\ell=0}^{j-1}\left( \tilde\Delta_{\H} - (\g -2\ell)^{2} \right) \left( \tilde \Delta_{\H} - (\g + 2\ell - 2\lfloor \g \rfloor )^{2} \right) \circ  y^{\frac{n}{2}-\g } \bigg|_{y=0}\\
  B_{ 2j + 2[\g]}^{2\g,\H }U=& \frac{1}{ b_{ 2j + 2[\g]}} \tilde B_{ 2j + 2[\g]}^{2\g,\H}U, \;\;\;\;\;\;\;\;0\leq j<\frac{1}{2}\lfloor\gamma\rfloor\\
  \tilde B_{2j + 2[\g]}^{2\g,\H}
  &= y^{-\frac{n}{2} + \g - 2j  - 2[\g]}\circ \prod_{\ell=0}^{j} \left( \tilde\Delta_{+} - (\g -2\ell)^{2} \right) \prod_{\ell=0}^{j - 1} \left( \tilde \Delta_{+} - (\g  + 2\ell - 2\lfloor \g \rfloor )^{2} \right) \circ y^{\frac{n}{2}-\g}  \bigg|_{y=0}
\end{align*}
As for $\B^{n+1}$, the definitions for $B_{2j}^{2\g,\H}$ and $B_{2j+2[g]}^{2\g,\H}$ for larger values of $j$ are defined analogously.

Next we state and prove how $\Delta_{+}$ acts on functions of the form $f\rho^{\a}$ for $\a \in \R$ and $f \in C^{\oo}(M)$.
This is needed to calculate the constants $b_{2j}$ and $b_{2j+2[\g]}$ and prove several properties of the $B_{\a}^{2\g}$.

\begin{lemma}
  \label{lm1.1-new}
  Let $f \in C^{\oo}(M)$ and $\a \in \R$.
  Then near $M$ there holds
  \begin{align*}
    \tilde\Delta_{+} (r^{\a}f) &= \left( \a - \frac{n}{2} \right)^{2} r^{\a}f + (2\a + 1 -n)\frac{W'}{W} r^{\a+1}f + \left( \frac{W''}{W}f + \mcM(r) f \right)r^{\a+2}.
  \end{align*}
\end{lemma}
\begin{proof}
  \begin{align*}
    \tilde\Delta_{+}  (r^{\a}f) &= W^{-1} W \circ \tilde\Delta_{+} \circ W^{-1} (W r^{\a}f) \\
    &= W^{-1} \left[ \left( r^{2} \p_{rr} + r^{2} \mcM(r) + (1-n)r \p_{r} + \frac{n^{2}}{4} \right)(Wr^{\a}f) \right]\\
    &= W^{-1} \left[ W'' r^{\a+2}f + 2\a W' r^{\a+1}f + \a(\a-1)W r^{\a}f + W r^{\a+2} \mcM(r)(f) \right.\\
    &\left. + (1-n)W' r^{\a+1}f + \a(1-n) W r^{\a}f + \frac{n^{2}}{4} Wr^{\a}f\right]\\
    &= \left( \a - \frac{n}{2} \right)^{2} r^{\a}f  + (2\a + 1 -n)\frac{W'}{W} r^{\a+1}f + \left( \frac{W''}{W}f + \mcM(r) (f) \right)r^{\a+2}.
  \end{align*}
\end{proof}
\begin{remark}
  Noting
  \begin{align*}
    W(r) &= 1 + W_{2}r^{2} + W_{4} r^{4} + \cdots\\
    W'(r) &= 2 W_{2} r + 4 W_{4}r^{3} + \cdots\\
    W''(r) &+ 2W_{2} + 12 W_{4} r^{2} + \cdots
  \end{align*}
  we have
  \[
    \tilde\Delta_{+} (r^{\a} f) = \left( \a - \frac{n}{2} \right)^{2}  r^{\a} f + O(r^{\a+2}).
  \]
\end{remark}

\begin{lemma}\label{lm1.1}
  Let $f \in C^{\oo}(M)$, let $\a \in \R$, let $\g \in (0,\frac{n}{2}) \setminus \N$ and let $\rho$ be $\g$-admissible.
  Then near $M$ there holds
  \begin{equation}
    \tilde\Delta_{+}(\rho^{\a} f) = \left( \a - \frac{n}{2} \right)^{2} \rho^{\a}f + O(\rho^{\a+2}).
    \label{eq:laplace-on-rho}
  \end{equation}
  In particular, for $m \in \R$, there holds
  \begin{equation}
    \left( \tilde\Delta_{+} - \g^{2} \right)(\rho^{\frac{n}{2}-\g}f)  = O(\rho^{\frac{n}{2} - \g + 2}).
    \label{eq:laplace-y-2m-on-rhof}
  \end{equation}
\end{lemma}
\begin{proof}
  This follows from the preceding remark and that $\rho$ is $\g$-admissible.
  To see this, let $x = \frac{\rho}{r} - 1$.
  Since $\rho$ is $\g$-admissible, we have by \eqref{eq:gamma-admissible-defining-function} that
  \begin{align*}
    \rho^{\a} &= \left( \frac{\rho}{r} \right)^{\a} r^{a}\\
    &= (1+x)^{\a} r^{\a}\\
    &= \left( 1 + \a x + \frac{\a(\a-1)}{2} x^{2} + \cdots  \right)r^{\a}\\
    &= \left( 1 + \a(\rho_{(2)}r^{2} + \cdots) + \frac{\a(\a-1)}{2} (\rho_{(2)}r^{2} + \cdots)^{2} + \cdots \right) r^{\a}\\
    &= r^{\a} + \a \rho_{(2)} r^{\a+2} + \cdots.
  \end{align*}
  Therefore,
  \begin{align*}
    \tilde\Delta_{+} (f \rho^{\a}) &= \tilde\Delta_{+} (r^{\a}f) + \a \tilde\Delta_{+} (\rho_{(2)}r^{\a+2}) f + \cdots\\
    &= \left( \a - \frac{n}{2} \right)^{2} r^{\a} f + \a \left( \a + 2 - \frac{n}{2} \right)^{2} \rho_{(2)} r^{\a+2} f + \cdots\\
    &= \left( \a - \frac{n}{2} \right)^{2} r^{\a}f + O(r^{\a+2})\\
    &= \left( \a - \frac{n}{2} \right)^{2} \rho^{a}f + O(\rho^{\a+2})
  \end{align*}
  where the last equality again uses $\g$-admissibility.
\end{proof}

The next lemma provides the calculations for computing $b_{2j}$ and $b_{2j+2[\g]}$.
It is also easy to see in the proof that
\begin{align*}
  b_{2j}&=0 \quad \text{ when } \quad \frac{1}{2}(1+\lfloor\gamma\rfloor)\leq j<1+ \lfloor\gamma\rfloor,\\
  b_{2j+2[\g]}&=0 \quad \text{ when } \quad \frac{1}{2}\lfloor\gamma\rfloor\leq j<1+ \lfloor\gamma\rfloor.
\end{align*}
Indeed, $ b_{2j} =0$ or $ b_{2j+2[\g]}=0$ when the factor ${\lfloor\gamma\rfloor-j-\ell}=0$.

\begin{lemma}\label{lm1.2}
  Given $j \in \N_{>0}$ and letting
  \begin{align*}
    \tilde B_{2j}^{2\g } &= \rho^{-\frac{n}{2} + \g - 2j}  \circ \prod_{\ell=0}^{j-1}\left( \tilde\Delta_{+} - (\g -2\ell)^{2} \right) \left( \tilde \Delta_{+} - (\g + 2\ell - 2\lfloor \g \rfloor )^{2}\right) \circ  \rho^{\frac{n}{2}-\g } |_{\rho=0}\\
    \tilde B_{2j + 2[\g]}^{2\g}
    &= \rho^{-\frac{n}{2} + \g - 2j  - 2[\g]}\circ \prod_{\ell=0}^{j} \left( \tilde\Delta_{+} - (\g -2\ell)^{2} \right) \prod_{\ell=0}^{j - 1} \left( \tilde \Delta_{+} - (\g  + 2\ell - 2\lfloor \g \rfloor )^{2} \right) \circ \rho^{\frac{n}{2}-\g}  |_{\rho=0},
  \end{align*}
  then there holds
  \begin{align*}
    b_{2j} := \tilde B_{2j}^{2\g } (\rho^{2j})=&4^{2j}j!\frac{\Gamma(j+1-[\gamma])}{\Gamma(1-[\gamma])}
    \frac{\Gamma(\gamma+1-j)}{\Gamma(\gamma+1-2j)}
    \frac{\Gamma(\lfloor \g \rfloor+1-j)}{\Gamma(\lfloor \g \rfloor+1-2j)};\\
    b_{2j+2[\g]} :=  \tilde B_{2j+2[\g]}^{2\g } (\rho^{2j+2[\g]})=&-4^{2j+1}j!\frac{\Gamma(j+1+[\gamma])}{\Gamma([\gamma])}
    \frac{\Gamma(\lfloor\gamma\rfloor+1-j)}{\Gamma(\lfloor\gamma\rfloor-2j)}
    \frac{\Gamma(\lfloor \g \rfloor+1-j-[\gamma])}{\Gamma(\lfloor \g \rfloor+1-2j-[\gamma])}.
  \end{align*}
\end{lemma}
\begin{proof}
  Using Lemma \ref{lm1.1}, we compute
  \begin{align*}
    b_{2j} =&\prod_{\ell=0}^{j-1}\left( (\gamma-2j)^{2} - (\g -2\ell)^{2} \right) \left(
    (\gamma-2j)^{2} - (\g + 2\ell - 2\lfloor \g \rfloor )^{2} \right)\\
    =&4^{2j}\prod_{\ell=0}^{j-1}(\ell-j)(\gamma-j-\ell)(\lfloor\gamma\rfloor-\ell-j)(\gamma+\ell-j-\lfloor\gamma\rfloor)\\
    =&4^{2j}\prod_{\ell=0}^{j-1}(\ell-j)(\gamma-j-\ell){(\lfloor\gamma\rfloor-\ell-j)}([\gamma]+\ell-j)\\
    =&4^{2j}j!\frac{\Gamma(j+1-[\gamma])}{\Gamma(1-[\gamma])}
    \cdot\frac{\Gamma(\gamma+1-j)}{\Gamma(\gamma+1-2j)}\cdot
    \frac{\Gamma(\lfloor \g \rfloor+1-j)}{\Gamma(\lfloor \g \rfloor+1-2j)}
  \end{align*}
  and
  \begin{align*}
    b_{2j+2[\g]} =&\prod_{\ell=0}^{j}\left( (\gamma-2j-2[\gamma])^{2} - (\g -2\ell)^{2} \right) \prod_{\ell=0}^{j-1}\left(
    (\gamma-2j-2[\gamma])^{2} - (\g + 2\ell - 2\lfloor \g \rfloor )^{2} \right)\\
    =&4^{2j+1}\prod_{\ell=0}^{j}(\ell-j-[\gamma])(\lfloor\gamma\rfloor-j-\ell)\prod_{\ell=0}^{j-1}(\ell-j)(\lfloor\gamma\rfloor-\ell-j-[\gamma])\\
    =&-4^{2j+1}\prod_{\ell=0}^{j}(j+[\gamma]-\ell){(\lfloor\gamma\rfloor-j-\ell)}\prod_{\ell=0}^{j-1}(j-\ell)(\lfloor\gamma\rfloor-\ell-j-[\gamma])\\
    =&-4^{2j+1}j!\frac{\Gamma(j+1+[\gamma])}{\Gamma([\gamma])}
    \cdot\frac{\Gamma(\lfloor\gamma\rfloor+1-j)}{\Gamma(\lfloor\gamma\rfloor-2j)}\cdot
    \frac{\Gamma(\lfloor \g \rfloor+1-j-[\gamma])}{\Gamma(\lfloor \g \rfloor+1-2j-[\gamma])}.
  \end{align*}

\end{proof}

In the next three lemmas we will state and prove important properties of the boundary operators which will be used in the other sections.
In a word, these results demonstrate that the boundary operators can be used to pick out the appropriate coefficients in the expansion for a function in $\mcC^{2\g}$.

\begin{lemma}
  For $j \in \N_{\geq0}$, there holds
  \begin{align*}
    B_{2j}^{2\g}f &= (\rho^{-2j}f)|_{\rho=0} \quad \text{ for } \quad f \in \rho^{2j} \ceven\\
    B_{2j + 2[\g]}^{2\g}f &= (\rho^{-2j-2[\g]}f)|_{\rho=0} \quad  \text{ for } \quad f \in \rho^{2j+2[\g]} \ceven.
  \end{align*}
  \label{lem:example-computations-of-B}
\end{lemma}

\begin{proof}
  If $f \in \rho^{2j} \ceven$ and $\lfloor \g \rfloor \in 2\N$, then there are boundary functions  $f_{\a} \in C^{\oo}(M)$ such that
  \[
    f = f_{2j} \rho^{2j} + \cdots + f_{\lfloor\g\rfloor+2j} \rho^{\lfloor \g\rfloor + 2j} + \psi_{*} \rho^{\lfloor \g \rfloor + 2j + 2[\g] } + o(\rho^{\lfloor \g \rfloor + 2j + 2[\g] })
  \]
  Evidently, by \eqref{eq:laplace-y-2m-on-rhof}, we obtain
  \begin{align*}
    \tilde B_{2j}^{2\g,\e}f &= \tilde B_{2j}^{2\g,\e}(\rho^{2j}f_{2j}) \\
    &=\tilde b_{2j}^{\e} f_{2j}.
  \end{align*}
  The result follows by dividing by $\tilde b_{2j}^{\e}$ and taking $\e \to 0$.
  The remaining identities may be similarly obtained.
\end{proof}

\begin{lemma}
  Let $f \in C^{\oo}(M)$ and $m \in \N_{\geq0}$.
  Then
  \begin{align*}
    B_{2j}^{2\g}(\rho^{2[\g]+2m}f) &=0 \quad \text { for } \quad 0 \leq j \leq \lfloor \g \rfloor \\
    B_{2j+2[\g]}^{2\g}(\rho^{2m}f) & = 0 \quad \text { for } \quad 0 \leq j \leq  \lfloor \g \rfloor.
  \end{align*}
  \label{lem:some-kernel-elements-of-B}
\end{lemma}

\begin{proof}
  To see
  \[
    B_{2j+2[\g]}^{2\g}(\rho^{2m}f)  = 0 \quad \text { for } \quad 0 \leq j < \frac{1}{2} \lfloor \g \rfloor,
  \]
  it is sufficient to consider $m=0$.
  Moreover,
  \[
    B_{2j}^{2\g}(\rho^{2[\g]+2m}f) =0 \quad \text { for } \quad 0 \leq j < \frac{1}{2}(1 + \lfloor \g \rfloor) \\
  \]
  will follow similarly.
  Lemma \ref{lm1.1} gives
  \begin{align*}
    \left( \tilde\Delta_{+} - \g ^{2} \right) \rho^{\frac{n}{2}-\g}f &= f_{0} \rho^{\frac{n}{2} - \g + 2} + O(\rho^{\frac{n}{2} - \g + 4}).
  \end{align*}
  for some $f_{0} \in C^{\oo}(\overline{X})$.
  Similarly,
  \begin{align*}
    \left( \tilde\Delta_{+} - (\g-2)^{2} \right) \left( \tilde\Delta_{+} - \g^{2} \right)  \rho^{\frac{n}{2} - \g}f = f_{1} \rho^{\frac{n}{2} - \g + 4} + O(\rho^{\frac{n}{2} - \g + 6})
  \end{align*}
  for some $f_{1} \in C^{\oo}(\overline{X})$.
  By induction, we have
  \begin{align*}
    \prod_{\ell=0}^{j}(\tilde\Delta_{+} - (\g-2\ell)^{2}) \rho^{\frac{n}{2} - \g} f = \rho^{\frac{n}{2} - \g  + 2j + 2} f_{*} + O(\rho^{\frac{n}{2} - \g + 2j + 2}).
  \end{align*}
  for some $f_{*} \in C^{\oo}(\overline{X})$.
  This is enough to conclude $B_{2j+2[\g]}^{2\g }f = 0$ for $0 \leq j < \frac{1}{2} \lfloor \g \rfloor$.

  To see that
  \begin{align*}
    B_{2j}^{2\g}(\rho^{2[\g]+2m}f) &= 0 \text{ for } \frac{1}{2}(1 + \lfloor \g \rfloor) \leq j \leq \lfloor \g \rfloor\\
    B_{2j+2[\g]}^{2\g}(\rho^{2m}f) &=0 \text{ for } \frac{1}{2}\lfloor \g \rfloor \leq j \leq \lfloor \g \rfloor,
  \end{align*}
  it is enough to observe
  \begin{align*}
     J_{+}^{2\lfloor \g \rfloor - 2j}\rho^{2m}f &=
    \begin{cases}
      0 & \lfloor \g \rfloor - j < m\\
      \rho^{2m} f & 0 \leq m \leq \lfloor \g \rfloor - j
    \end{cases}\\
    J_{-}^{2\lfloor \g \rfloor - 2j} \rho^{2[\g]+2m}f &=
     \begin{cases}
       0 & \lfloor \g \rfloor - j < m\\
       \rho^{2[\g]+2m}f & 0 \leq m \leq \lfloor \g \rfloor
     \end{cases}.
  \end{align*}
  and
  \begin{align*}
    \tilde B_{2j}^{2\g,\e}\rho^{2m+2[\g]}f &= 0 \quad \text{ for }\quad  \frac{1}{2}(1 + \lfloor \g \rfloor) \leq j \leq \lfloor \g \rfloor, \quad \lfloor \g \rfloor - j <m\\
    \tilde B_{2j+2[\g]}^{2\g,\e}\rho^{2m}f &= 0 \quad \text{ for }\quad \frac{1}{2}\lfloor \g \rfloor \leq j \leq \lfloor \g \rfloor, \quad \lfloor \g \rfloor - j < m.
  \end{align*}

\end{proof}

Lemmas \ref{lem:example-computations-of-B} and \ref{lem:some-kernel-elements-of-B} allow us to write $f \in \mathcal{C}^{2\g}$ as a series in terms of the boundary operators.
This is detailed in the following corollary.

\begin{corollary}
  If $f \in \mathcal{C}^{2\g}$, then in the expansion
  \[
    f = \sum_{\ell=0}^{\oo} \rho^{ 2\ell} f _{ 2\ell} + \sum_{\ell=0}^{\oo} \rho^{2\ell + 2[\g]} f_{ 2\ell + 2[\g] }, \quad f_{\a} \in C^{\oo}(M),
  \]
  we may take
  \begin{align*}
    f_{ 2\ell } &= B_{ 2\ell }^{2\g}\left( f - \sum_{m=0}^{\ell-1} \rho^{ 2m } f_{ 2m } \right)\\
    f_{ 2\ell + 2[\g]} &= B_{ 2\ell + 2[\g]}^{2\g}\left( f - \sum_{m=0}^{\ell-1} \rho^{ 2m + 2[\g]} f_{ 2m + 2[\g]} \right)
  \end{align*}
  for $0 \leq j \leq \lfloor \g \rfloor$.
  \label{cor:general-series-expansion-in-terms-of-B}
\end{corollary}

\begin{proof}
  Since $f \in \mathcal{C}^{2\g}$, there are functions $\left\{ h_{j},g_{j} \right\} \subset C^{\oo}(M)$ such that
  \[
    f = \sum_{j=0}^{\oo} \rho^{2j}f_{2j} + \sum_{j=0}^{\oo} \rho^{2j+2[\g]}f_{2j + 2[\g]}.
  \]
  Evidently, Lemmas \ref{lem:example-computations-of-B} and \ref{lem:some-kernel-elements-of-B} give
  \[
    B_{0}^{2\g} f = f_{0} \text{ and } B_{2[\g]}^{2\g} f = f_{2[\g]}.
  \]
  Similarly,
  \[
    B_{2}^{2\g}(f - B_{0}^{2\g} f)  = f_{2} \text{ and } B_{2[\g] + 2}^{2\g} (f - \rho^{2[\g]} B_{2[\g]}^{2\g}f) = f_{2 + 2[\g]}.
  \]
  Proceeding inductively gives the desired series expansion.

\end{proof}

\section{Dirichlet Problem and Scattering}
\label{sec:dirichlet-problem-and-scattering}

Scattering theory was initially studied by Mazzeo and Melrose \cite{MM} and was further developed by Graham and Zworski \cite{MR1965361}
to derive the fractional GJMS operators based on the construction of the ambient space of C. Fefferman and Graham \cite{FeffermanGr2, FeffermanGr}.

Fix $\g \in (0,\frac{n}{2})\setminus\N$ and set $k = \lfloor \g \rfloor + 1$.
Throughout, we fix a conformal infinity $g \in [g]$ on $M$, let $r$ denote the resulting geodesic defining function and fix a $\g$-admissible defining function $\rho$.
We begin recall the scattering theory on Poincar\'e-Einstein manifolds.
Suppose $\frac{n^{2}}{4} - \g^{2}$ is not in the $L^{2}$-spectrum of $-\Delta_{+}$.
Given $f \in C^{\oo}(M) \cap H^{\g,2}(M)$ and $s =\frac{n}{2} + \g$, there is a unique solution $\mathcal{P}(s)f$ of the Poisson equation
\begin{equation}
  \Delta_{+} V + s(n-s)V = 0
  \label{eq:scattering-equation}
\end{equation}
such that, asymptotically near $M$, there holds
\begin{equation}
  \begin{cases}
    \mathcal{P}(s)f =  r^{n-s}F +  r^{s}G & \text{ for some } F,G \in C^{\oo}(\overline{X})\\
    F|_{M} = f\\
  \end{cases}.
  \label{eq:scattering-equation-expansion}
\end{equation}
Defining the scattering operator $S(s):f  \mapsto  G|_{M}$, we recall that the operator $S$ defines the fractional GJMS operators $P_{\g}$ on $M$ by
\begin{equation}
  P_{\g} f =c_{\g}S(s)f, \quad  c_{\g} = 2^{\g} \frac{\Gamma(\g)}{\Gamma(-\g)}
  \label{eq:fractional-boundary-operator}
\end{equation}
Therefore, by (\ref{eq:scattering-equation-expansion}), we obtain
\begin{align}\nonumber
  \mathcal{P}(\frac{n}{2}+\gamma)f=& r^{\frac{n}{2}-\gamma}(f+ r^{2}f_{1}+\cdots)+\\
  \label{b2.4}
  &2^{-\g} \frac{\Gamma(-\g)}{\Gamma(\g)} r^{\frac{n}{2}+\gamma}\left[ P_{\g} f + r^{2}g_{1}+\cdots\right]
\end{align}

Recall that we introduced the following weighted operators:
\[
  D_{s} = -\Delta_{+} - s(n-s), \quad L_{2k}^{+} = \prod_{j=0}^{k-1}D_{s-j}, \quad s = \frac{n}{2} + \g
\]
and
\[
  L_{2k} = \rho^{-\frac{n}{2} + \g - 2k} \circ L_{2k}^{+} \circ \rho^{\frac{n}{2} - \g}.
\]
To realize the boundary operators as generalized Dirichlet-to-Neumann operators in the spirit of Caffarelli-Silvestre, we
will study the following Dirichlet problem.
\begin{equation}
  \begin{cases}
    L_{2k} V = 0 & \text{in }X\\
    B_{2j}^{2\g}V = f^{(2j)},& 0 \leq j \leq [\g/2]\\
    B_{2j+2[\g]}^{2\g}V = \phi^{(2j)}, & 0 \leq j \leq [\g] - [\g/2] -1
  \end{cases}.
  \label{eq:boundary-value-dirichlet-problem-l2k}
\end{equation}
Here, $f^{(2j)} \in C^{\oo}(M) \cap H^{\g-2j}(M)$ and $\phi^{(2j)} \in C^{\oo}(M) \cap H^{\lfloor \g \rfloor - [\g] - 2j}(M)$ are the boundary data.

\begin{theorem}
  Let $\g \in (0,\oo) \setminus \N$ and fix boundary data
  \begin{align*}
    f^{(2j)} \in C^{\oo}(M) \cap H^{\g-2j}(M) \qquad   &\text{ for } \qquad 0 \leq j \leq \lfloor \g/2 \rfloor\\
    \phi^{(2j)}  \in C^{\oo}(M) \cap H^{\lfloor \g \rfloor - [\g] -2j}(M) \qquad & \text{ for } \qquad 0 \leq j \leq \lfloor \g \rfloor - \lfloor \g/2 \rfloor - 1.
  \end{align*}
  Then the Dirichlet problem
  \begin{equation}
    \begin{cases}
      L_{2k} V = 0, & \text{in }X\\
      B_{2j}^{2\g}V = f^{(2j)},& 0 \leq j \leq \lfloor \g/2 \rfloor\\
      B_{2j+2[\g]}^{2\g}V = \phi^{(2j)}, & 0 \leq j \leq \lfloor \g \rfloor - \lfloor \g/2 \rfloor -1
    \end{cases}.
    \label{eq:boundary-value-dirichlet-problem-l2k-copied}
  \end{equation}
  has a unique solution given by
  \begin{equation}
    V = \sum_{j=0}^{\lfloor \g/2 \rfloor} \rho^{-\frac{n}{2}+\g} \mathcal{P}\left( \frac{n}{2}+\g-2j \right)f^{(2j)}  + \sum_{j=0}^{\lfloor \g \rfloor - \lfloor \g/2 \rfloor - 1} \rho^{-\frac{n}{2} + \g} \mathcal{P}\left( \frac{n}{2} + \lfloor \g \rfloor - [\g] - 2j \right) \phi^{(2j)}.
    \label{eq:V-expansion-in-terms-of-f-and-phi}
  \end{equation}
  \label{thm:solution-to-l2k-dirichlet-problem}
\end{theorem}

\begin{proof}
  Set
  \[
    V = \sum_{j=0}^{\lfloor \g/2 \rfloor} \rho^{-\frac{n}{2}+\g} \mathcal{P}\left( \frac{n}{2}+\g-2j \right)f^{(2j)}  + \sum_{j=0}^{\lfloor \g \rfloor - \lfloor \g/2 \rfloor - 1} \rho^{-\frac{n}{2} + \g} \mathcal{P}\left( \frac{n}{2} + \lfloor \g \rfloor - [\g] - 2j \right) \phi^{(2j)}.
  \]
  By definition of $L_{2k}$, we have
  \[
    L_{2k}V = (-1)^{k}\rho^{-\frac{n}{2} + \g - 2k} \prod_{j=0}^{k-1}\left( \tilde\Delta_{+} - (\g - 2j)^{2} \right)(\rho^{\frac{n}{2}-\g}V)
  \]
  and so it is easy to see that $L_{2k} V$ since
  \[
    \left( \tilde\Delta_{+} - \mu^{2} \right)\mcP(\frac{n}{2} + \mu )f = 0
  \]
  for suitable $\mu$ and $f \in C^{\oo}(M)$.
  We now verify that $V$ satisfies the boundary conditions in \eqref{eq:boundary-value-dirichlet-problem-l2k-copied}.
  Let $F_{j},G_{j} \in C^{\oo}(\overline{X})$ be such that
  \[
    \mathcal{P}\left( \frac{n}{2}+\g-2j \right)f^{(2j)} = \rho^{\frac{n}{2} - \g + 2j} F_{j} + \rho^{\frac{n}{2} + \g - 2j} G_{j}.
  \]
  with $F_{j}|_{\rho=0} = f^{(2j)}$, from which we may write
  \[
    \rho^{-\frac{n}{2} + \g} \mathcal{P}\left( \frac{n}{2}+\g-2j \right)f^{(2j)} = \rho^{2j}F_{j} + \rho^{2\g-2j}G_{j}.
  \]
  Writing $F_{j}$ as a $\ceven$ function$\mod O(\rho^{\oo})$, we have
  \begin{equation}
    \rho^{-\frac{n}{2} + \g} \mathcal{P}\left( \frac{n}{2}+1+\g-2j \right)f^{(2j)} = \rho^{2j}f^{(2j)} + \rho^{2j+2}f_{*} + \cdots \mod O(\rho^{\oo})
    \label{eq:expansion-mod-rho-infinity-for-f2j}
  \end{equation}
  for some $f_{*} \in C^{\oo}(M)$.
  We may also expand
  \[
    \rho^{-\frac{n}{2} + \g} \mathcal{P}\left( \frac{n}{2}+1+\lfloor \g \rfloor - [\g] - 2j \right) \phi^{(2j)}
  \]
  in a similar fashion.
  Using Lemmas \ref{lem:example-computations-of-B} and \ref{lem:some-kernel-elements-of-B} or that $B_{2j}^{2\g}$ has a factor of the form $\tilde\Delta_{+} - \mu^{2}$ that annihilates either
  \[
    \mathcal{P}\left( \frac{n}{2}+\g-2m \right)f^{(2m)}
  \]
  or
  \[
    \mathcal{P}\left( \frac{n}{2}+1+\lfloor \g \rfloor - [\g] - 2m' \right)\phi^{(2j)},
  \]
  it is easy to see that
  \[
    B_{2j}^{2\g}V = B_{2j}^{2\g}\left( \rho^{-\frac{n}{2} + \g} \mathcal{P}\left( \frac{n}{2} + \g - 2j \right) \right) = f^{(2j)}.
  \]
  Using similar reasoning, we also obtain
  \[
    B_{2j+2[\g]}^{2\g} V = \phi^{(2j)}.
  \]

  We now show that the solution is unique.
  First note that, if $U \in C^{2\g}$ solves
  \[
    \left( \tilde\Delta_{+} - \mu^{2} \right)(\rho^{\frac{n}{2}-\mu}U)=0
  \]
  with $U|_{\rho=0} = 0$, then $U = 0$ by uniqueness of the scattering problem.
  Now consider the case $\lfloor \g \rfloor = 2P \in 2\N$, noting that the case $\lfloor \g \rfloor \in 2\N + 1$ is handled similarly.
 It is easy to see that \eqref{eq:boundary-value-dirichlet-problem-l2k-copied} is equivalent
  \[
    \begin{cases}
      \left( \tilde\Delta_{+} - (\g - 2P)^{2} \right) \prod_{\ell=0}^{P-1}\left( \tilde\Delta_{+}  - (\g - 2\ell)^{2} \right)\left( \tilde\Delta_{+} - \left( \g +2 \ell - 2\lfloor \g \rfloor \right)^{2} \right)(\rho^{\frac{n}{2}-\g}V) = 0\\
      B_{2j}^{2\g}V = f^{(2j)} ,\qquad 0 \leq j \leq \lfloor \g/2 \rfloor \\
      B_{2j+2[\g]}^{2\g} (V) = \phi^{(2j)} ,\qquad 0 \leq j \leq \lfloor \g \rfloor - \lfloor \g /2 \rfloor -1
    \end{cases}.
  \]
  Letting $V'$ be another solution and setting
  \[
    U = \rho^{-\frac{n}{2} + \g - 2P} \prod_{\ell=0}^{P-1}\left( \tilde\Delta_{+}  - (\g - 2\ell)^{2} \right)\left( \tilde\Delta_{+} - \left( \g +2 \ell - 2\lfloor \g \rfloor \right)^{2} \right)(\rho^{\frac{n}{2}-\g}(V-V')),
  \]
  we conclude
  \[
    \left( \tilde\Delta_{+} - (\g-2P)^{2} \right)(\rho^{\frac{n}{2}-\g+2P}U) = 0
  \]
  and that
  \[
    U|_{\rho=0} = B_{2P}^{2\g}(V) = 0.
  \]
  Since this implies $U =0$, we conclude $V-V'$ solves
  \[
    \begin{cases}
      \prod_{\ell=0}^{P-1}\left( \tilde\Delta_{+}  - (\g - 2\ell)^{2} \right)\left( \tilde\Delta_{+} - \left( \g +2 \ell - 2\lfloor \g \rfloor \right)^{2} \right)(\rho^{\frac{n}{2}-\g}(V-V')) = 0\\
      B_{2j}^{2\g}(V-V') = 0 \\
      B_{2j+2[\g]}^{2\g} (V-V') = 0
    \end{cases}.
  \]
  Induction gives $V=V'$.
\end{proof}

\section{Proofs of Main Results}
\label{sec:proofs-section}

In this section we prove Theorems \ref{thm:boundary-operator-conformal-covariance}, \ref{thm:dirichlet-form-symmetry}, \ref{thm:boundary-to-fractional-operator} and \ref{th1.5}, as well as two important integral identities.
In order to prove the symmetry of the Dirichlet form (Theorem \ref{thm:dirichlet-form-symmetry}) and trace inequalities (Theorem \ref{th1.5}), we establish a Green-type identity on $\overline{X}$ so as to access the appropriate boundary integrals.
We also establish a highly nontrivial integral identity (Theorem \ref{thm:main-integral-identity} stated below) relating the term
\[
  \int_{X} U L_{2k}V \cdot \rho^{1-2[\g]} dvol_{\rho^2 g_+}
\]
appearing in the Dirichlet form $\mathcal{Q}_{2\g}(U,V)$ to the appropriate boundary integrals.
The proof of Theorem \ref{thm:main-integral-identity} contains the majority of the work necessary for establishing the higher order trace inequalities and the symmetry $\mathcal{Q}_{2\g}(U,V) = \mathcal{Q}_{2\g}(V,U)$.

\subsection*{Proof of Theorem \ref{thm:boundary-operator-conformal-covariance}}
We aim to prove the conformal covariance of our boundary operators.
Consider the conformal change of defining function: $\hat\rho = e^{\tau} \rho$.
Note that $\mcC^{2\g}(X)$ defined relative to $\rho$ and $\hat\rho$ are one and the same.
Let $\widehat{B}_{2j}^{2\gamma}$ and $\widehat{B}_{2j+2[\gamma]}^{2\g}$ be the boundary operators associated with $(X, \widehat{\rho}^{2}g_{+})$.
Let $\hat\rho = e^{\tau} \rho$.
Context will make it clear when $\tau$ is supposed to be restricted to $M$ and so we forgo the notation $\tau|_{M}$ in this proof.
For the indicated index ranges, we may write
\begin{align*}
  \hat B_{2 j}^{2\g}U=&\frac{1}{ b_{ 2j }}\hat {\tilde B}_{ 2j}^{2\g}U,\;\; \;\;\;\;\;\;\;\;\;\;\;\;\;\;\;\;\;\;0\leq j<\frac{1}{2}(1+\lfloor\gamma\rfloor); \\
  \hat B_{ 2j + 2[\g]}^{2\g }U=& \frac{1}{ b_{ 2j + 2[\g]}} \hat {\tilde B}_{ 2j + 2[\g]}^{2\g}U, \;\;\;\;\;\;\;\;0\leq j<\frac{1}{2}\lfloor\gamma\rfloor,
\end{align*}
where
\begin{align*}
  \hat{\tilde B}_{2j}^{2\g } &= \hat\rho^{-\frac{n}{2} + \g - 2j}  \circ \prod_{\ell=0}^{j-1}\left( \tilde\Delta_{+} - (\g -2\ell)^{2} \right) \left( \tilde \Delta_{+} - (\g + 2\ell - 2\lfloor \g \rfloor )^{2} \right) \circ  \hat \rho^{\frac{n}{2}-\g } |_{\hat\rho=0};\\
  \hat{\tilde B}_{2j + 2[\g]}^{2\g}
  &= \hat \rho^{-\frac{n}{2} + \g - 2j  - 2[\g]}\circ \prod_{\ell=0}^{j} \left( \tilde\Delta_{+} - (\g -2\ell)^{2} \right) \prod_{\ell=0}^{j - 1} \left( \tilde \Delta_{+} - (\g  + 2\ell - 2\lfloor \g \rfloor )^{2} \right) \circ \hat\rho^{\frac{n}{2}-\g}  |_{\hat \rho=0}.
\end{align*}
Using
\begin{align*}
  \hat\rho^{-\frac{n}{2} + \g - 2j} &= e^{(-\frac{n}{2} + \g - 2j)\tau}\rho^{-\frac{n}{2} + \g - 2j}\\
  \hat\rho^{-\frac{n}{2} + \g - 2j - 2[\g]} &= e^{(-\frac{n}{2} + \g - 2j - 2[\g])\tau} \rho^{-\frac{n}{2} + \g - 2j - 2[\g]}\\
  \hat\rho^{\frac{n}{2} - \g} &= e^{(\frac{n}{2} - \g)\tau} \rho^{\frac{n}{2} - \g},
\end{align*}
it is easy to verify then
\begin{align*}
  \hat{B}_{2j}^{2\g}U &= e^{(-\frac{n}{2} + \g - 2j) \tau} B_{2j}^{2\g}(e^{(\frac{n}{2} - \g) \tau}U) \quad \text{ for } \quad 0 \leq j \leq \frac{1}{2}(1 + \lfloor \g \rfloor)\\
  \hat{B}_{2j+2[\g]}^{2\g}U &= e^{(-\frac{n}{2} + \g -2j - 2[\g]) \tau} B_{2j+2[\g]}^{2\g}(e^{(\frac{n}{2} - \g) \tau}U) \quad \text{ for } \quad 0 \leq j \leq \frac{1}{2} \lfloor \g \rfloor.
\end{align*}
From these identities, it is also easy to see that the jets $J_{\pm}^{2m}$ are also conformally covariant in the sense that
\[
  \hat{J}_{\pm}^{2m}(U) = e^{(-\frac{n}{2} + \g)\tau} J_{\pm}^{2m}(e^{(\frac{n}{2} - \g)\tau}U),
\]
where $m$ is such that the $B_{\a}^{2\g}$ appearing in the expression of $J_{\pm}^{2m}$ are well-defined.
From this we conclude
\begin{align*}
  \hat B_{2j}^{2\g} U &= \hat{\overline{B}}_{2j}^{2\g}(U - \hat J_{-}^{2\lfloor \g \rfloor - 2j}U)\\
  &= e^{(-\frac{n}{2} + \g - 2j)\tau}\overline{B}_{2j}^{2\g}(e^{(\frac{n}{2} - \g)\tau}U - J_{-}^{2\lfloor \g \rfloor - 2j}e^{(\frac{n}{2} - \g)\tau} U) \\
  &= e^{(-\frac{n}{2} + \g - 2j)\tau}B_{2j}^{2\g}(e^{(\frac{n}{2} - \g )\tau}U)
\end{align*}
for $\frac{1}{2} (1 + \lfloor \g \rfloor) \leq j \leq \lfloor \g \rfloor$ and
\begin{align*}
  \hat B_{2j+2[\g]}^{2\g} U &= \hat{\overline{B}}_{2j+2[\g]}^{2\g} (U - \hat J_{+}^{2\lfloor \g \rfloor - 2j}U) \\
  &= e^{(-\frac{n}{2} + \g - 2j - 2[\g])}\overline{B}_{2j+2[\g]}^{2\g}(e^{(\frac{n}{2} - \g)\tau} U - J_{+}^{2\lfloor \g \rfloor - 2j}e^{(\frac{n}{2} - \g)}U)\\
  &=e^{(-\frac{n}{2} + \g - 2j - 2[\g])\tau}B_{2j + 2[\g]}^{2\g}(e^{(\frac{n}{2} - \g ) \tau} U)
\end{align*}
for $\frac{1}{2}\lfloor \g \rfloor \leq j \leq \lfloor \g \rfloor$.

\qed

\subsection*{Proof of Theorem \ref{thm:boundary-to-fractional-operator}}

Next, we show that our boundary operators recover the fractional conformally covariant operators.
If we set $ B_{2j}^{2\g}(U) = f^{(2j)}$ and $B_{2j+2[\g]}^{2\g}(U) = \phi^{(2j)}$, then $U$ is exactly the  function given by (\ref{eq:V-expansion-in-terms-of-f-and-phi}).
By using (\ref{b2.4}), we have
\begin{align*}
  \rho^{-\frac{n}{2}+\g} \mathcal{P}\left( \frac{n}{2}+\g-2j \right)f^{(2j)} =&\rho^{2j}(f^{(2j)}+\cdots)+2^{2j-\g} \frac{\Gamma(2j-\g)}{\Gamma(\g-2j)}\rho^{2\gamma-2j}\left( P_{\g-2j} f^{(2j)} +\cdots\right).
\end{align*}
But
\begin{align*}
  B^{2\gamma}_{2\gamma-2j}(\rho^{2k})=&0 \text{ for }k=1,2,\ldots
\end{align*}
and so
\begin{align*}
  B^{2\gamma}_{2\gamma-2j}(U)=B^{2\gamma}_{2\gamma-2j}(\rho^{2\gamma-2j})2^{2j-\g} \frac{\Gamma(2j-\g)}{\Gamma(\g-2j)}P_{\g-2j} f^{(2j)}=2^{2j-\g} \frac{\Gamma(2j-\g)}{\Gamma(\g-2j)}P_{\g-2j} f^{(2j)}.
\end{align*}
Here we use the fact
\begin{align*}
  B^{2\gamma}_{2\gamma-2j}\left( \rho^{-\frac{n}{2}+\g} \mathcal{P}\left( \frac{n}{2}+\g-2k \right)f^{(2k)}\right)=&0,\;\;k\neq j
\end{align*}
because $\mathcal{P}\left( \frac{n}{2}+1+\g-2k \right)f^{(2k)}$ satisfies
\begin{align*}
  \left(\tilde \Delta_{+} -(\gamma-2k)^{2}\right)\mathcal{P}\left( \frac{n}{2}+\g-2k \right)f^{(2k)}   = 0.
\end{align*}

Similarly,
\begin{align*}
  &\rho^{-\frac{n}{2} + \g} \mathcal{P}\left( \frac{n}{2} + \lfloor \g \rfloor - [\g] - 2j \right) \phi^{(2j)}\\
  =&\rho^{2j+2[\gamma]}(\phi^{(2j)}+\cdots)+2^{2j+[\g]-\lfloor \g \rfloor } \frac{\Gamma(2j+[\g]-\lfloor \g \rfloor)}{\Gamma(\lfloor \g \rfloor-2j-[\g])}\rho^{2\lfloor \g \rfloor-2j}\left( P_{\lfloor \g \rfloor - [\g] - 2j} \phi^{(2j)} +\cdots\right)
\end{align*}
and
\begin{align*}
  B^{2\gamma}_{2\lfloor\gamma\rfloor-2j}(U)=&2^{2j+[\g]-\lfloor \g \rfloor} \frac{\Gamma(2j+[\g]-\lfloor \g \rfloor)}{\Gamma(\lfloor \g \rfloor-2j-[\g])}
  P_{\lfloor \g \rfloor - [\g] - 2j} \phi^{(2j)}.
\end{align*}

\qed

\subsection*{Main Integral Identity}
We now move to proving the symmetry of $\mathcal{Q}_{2\g}$ and the higher order trace inequalities.
Before doing so, we establish the following Green-type identity and the integral identity in Theorem \ref{thm:main-integral-identity}.

\begin{lemma}
  \[
    \int_{X} u \Delta_{+} v dvol_{g_{+}} - \int_{X}v \Delta_{+}u  dvol_{g_{+}} = \int_{M} \left( r^{1-n} u \p_{r} v - r^{1-n} v \p_{r}u \right)|_{r=0} dvol_{g}.
  \]
  \label{lem:green-identity}
\end{lemma}

\begin{proof}
  Apply the Greens identity on the bounded subset of $X$ whose boundary is given by a level set of $r$.
  Then take the limit as $r \to 0$.
\end{proof}

Now, given $U,V \in \mcC^{2\g} \cap \dot H^{k,\g}(X)$, let $\tilde U,\tilde V$ be the respective solutions of the Dirichlet problem
\begin{equation*}
  \begin{cases}
    L_{2k} \tilde{W} = 0 & \text{ in } X\\
    B_{2j}^{2\g}\tilde W = B_{2j}^{2\g} W & 0 \leq j \leq \lfloor \g/2 \rfloor\\
    B_{2j + 2[\g]}^{2\g} \tilde W = B_{2j+2[\g]}^{2\g} W & 0 \leq j \leq \lfloor \g \rfloor - \lfloor \g/2 \rfloor - 1
  \end{cases}
\end{equation*}
for $W=U,V$.
For notational simplicity, we assume $\lfloor \g \rfloor \in 2 \N_{>0}$ so that we may freely write $\lfloor \g /2 \rfloor = \lfloor \g \rfloor /2$; the case $\lfloor \g \rfloor \in 2 \N_{>0} - 1$ is treated identically.
From Theorem \ref{thm:solution-to-l2k-dirichlet-problem} we have
\begin{equation}
  \tilde U = \sum_{j=0}^{\lfloor \g/2 \rfloor} U_{j} + \sum_{j=0}^{\lfloor \g \rfloor - \lfloor \g/2 \rfloor -1} U_{j}',
  \label{eq:expansion-for-u-tilde}
\end{equation}
where
\begin{align*}
  U_{j} &= \rho^{-\frac{n}{2} + \g} \mcP\left( \frac{n}{2}+\g-2j \right) B_{2j}^{2\g}U\\
  U_{j}' &= \rho^{-\frac{n}{2} + \g} \mcP \left( \frac{n}{2}+ \lfloor \g \rfloor - [\g] - 2j \right) B_{2j + 2[\g]}^{2\g} U.
\end{align*}
For $j=0,\ldots,\lfloor \g \rfloor - \lfloor \g/2 \rfloor -1$, set $U_{\lfloor \g/2 \rfloor +j + 1} = U_{\lfloor \g \rfloor - \lfloor \g/2 \rfloor -1 -j}'$ and so
\[
  U_{\lfloor \g/2 \rfloor + j + 1} =  \rho^{-\frac{n}{2} +\g} \mcP \left(  \frac{n}{2}  - \g + 2 ( \lfloor \g/2 \rfloor + j + 1) \right) B_{2 \g - 2\lfloor \g/2 \rfloor - 2 -2j}^{2\g}U.
\]
It follows that
\begin{equation}
  \left( \tilde\Delta_{+} - (\g - 2j)^{2} \right) \rho^{\frac{n}{2}-\g}U_{j} = 0, \qquad j = 0, \ldots, \lfloor \g \rfloor .
  \label{eq:uj-satisfies-scattering-equation}
\end{equation}
By scattering theorem (see Section \ref{sec:dirichlet-problem-and-scattering}), we have
\begin{equation}
  U_{j}=
  \begin{cases}
    \rho^{2j}F_{j} + \rho^{2\g-2j}G_{j} & j = 0 , \ldots, \lfloor \g /2 \rfloor\\
    \rho^{2\g - 2j}F_{j} + \rho^{2j}G_{j} & j= \lfloor \g /2 \rfloor + 1, \ldots, \lfloor \g \rfloor
  \end{cases},
  \label{eq:scattering-uj}
\end{equation}
where $F_{j},G_{j} \in C^{\oo}(\overline{X})$ and
\begin{equation}
  \begin{aligned}
    F_{j}|_{\rho=0} &=
    \begin{cases}
      B_{2j}^{2\g}U & j=0, \ldots, \lfloor \g/2 \rfloor\\
      B_{2\g-2j}^{2\g}U& j= \lfloor \g/2\rfloor + 1,\ldots, \lfloor \g \rfloor
    \end{cases}\\
    G_{j}|_{\rho=0} & =
    \begin{cases}
      S\left( \frac{n}{2}+\g-2j \right) B_{2j}^{2\g}U & j=0, \ldots, \lfloor \g/2 \rfloor\\
      S \left(  \frac{n}{2} + \g - 2 j \right) B_{2\g-2j}^{2\g} U & j= \lfloor \g/2\rfloor + 1,\ldots, \lfloor \g \rfloor
    \end{cases}.
  \end{aligned}
  \label{eq:fj-gj-boundary-values}
\end{equation}
We can prepare all the same for $V$ and $\tilde V$.

We now state and prove the main integral identity.

\begin{theorem}\label{thm:main-integral-identity}
  Letting $U$, $\tilde U$, $V$ and $\tilde V$ be as above, there holds
  \begin{align*}
    &\int_{X} U L_{2k}V \cdot \rho^{1-2[\g]} dvol_{\rho^2 g_+} = \int_{X} \rho^{\frac{n}{2}-\g} (U - \tilde U) L_{2k}^{+}( \rho^{\frac{n}{2}-\g} (V-\tilde V)) dvol_{g}\\
    &+\sum_{j=0}^{\lfloor \g/2 \rfloor} \sigma_{j,\g}  \int_{M} B_{2j}^{2\g}U  B_{2\g-2j}^{2\g}V  dvol_{g} + \sum_{j=\lfloor \g/2 \rfloor + 1}^{\lfloor \g \rfloor} \sigma_{j,\g} \int_{M}   B_{2\g-2j}^{2\g}U  B_{2j}^{2\g}V  dvol_{g}\\
    &+ \sum_{j=0}^{\lfloor \g/2 \rfloor} \varsigma_{j,\g} \int_{M}   B_{2j}^{2\g}U P_{\g-2j}B_{2j}^{2\g} V    dvol_{g}+ \sum_{j=\lfloor \g /2 \rfloor + 1}^{\lfloor \g \rfloor} \varsigma_{j,\g} \int_{M}   B_{2\g - 2j}^{2\g}U P_{2j - \g}B_{2\g-2j}^{2\g} V    dvol_{g},
  \end{align*}
  where
  \begin{align*}
    \sigma_{j,\g} &=
    \begin{cases}
      2 \frac{\Gamma(\g-j+1)\Gamma(j+1)\Gamma(j+\lfloor \g \rfloor - \g + 1)\Gamma(\lfloor \g \rfloor - j + 1)}{\Gamma(\g-2j) \Gamma(2j-\g+1)} & j = 0,\ldots, \lfloor \g/2 \rfloor\\
      2 \frac{\Gamma(\g-j+1)\Gamma(j+1)\Gamma(j+\lfloor \g \rfloor - \g + 1)\Gamma(\lfloor \g \rfloor - j + 1)}{\Gamma(\g-2j + 1) \Gamma(2j-\g)} & j = \lfloor \g /2 \rfloor + 1,\ldots, \lfloor \g \rfloor\\
    \end{cases}\\
    \varsigma_{j,\g} &=
    \begin{cases}
      2^{2j-\g+1} \frac{\Gamma(\g-j+1)\Gamma(j+1)\Gamma(j+\lfloor \g \rfloor - \g + 1)\Gamma(\lfloor \g \rfloor - j + 1)}{\Gamma(\g-2j)\Gamma(\g-2j+1)} & j = 0,\ldots, \lfloor \g /2 \rfloor\\
      2^{\g-2j+1} \frac{\Gamma(\g-j+1)\Gamma(j+1)\Gamma(j+\lfloor \g \rfloor - \g +1)\Gamma(\lfloor \g \rfloor - j + 1)}{\Gamma(2j-\g)\Gamma(2j-\g+1)} & j = \lfloor \g/2 \rfloor + 1, \ldots, \lfloor \g \rfloor
    \end{cases}
    .
  \end{align*}
\end{theorem}

\begin{proof}

  Using
  \[
    L_{2k}^{+}(\rho^{\frac{n}{2}-\g}\tilde V) = 0,
  \]
  we get
  \begin{align*}
    \int_{X} U L_{2k}V \cdot \rho^{1-2[\g]} dvol_{\rho^2 g_+} &= \int_{X} \rho^{\frac{n}{2}-\g} U L_{2k}^{+} (\rho^{\frac{n}{2}-\g}V) dvol_{g}\\
    &= \int_{X} \rho^{\frac{n}{2}-\g} (U - \tilde U) L_{2k}^{+}( \rho^{\frac{n}{2}-\g} (V-\tilde V)) dvol_{g}\\
    &+ \int_{X} \rho^{\frac{n}{2}-\g} \tilde U L_{2k}^{+} (\rho^{\frac{n}{2}-\g}(V- \tilde V))dvol_{g}\\
    &=: A_{1} + A_{2}.
  \end{align*}
  We will use the expansions \eqref{eq:expansion-for-u-tilde} for $\tilde U$ and $\tilde V$ and the Green-type identity Lemma \ref{lem:green-identity} to rewrite $A_{2}$ in terms of boundary integrals involving the boundary operators $B_{\a}^{2\g}$.
  For notational convenience and consistency, we introduce the following notation:
  \begin{align*}
    \Pi^{\g} &= (-1)^{\lfloor \g \rfloor +1}\prod_{i=0}^{\lfloor \g \rfloor} \left( \tilde \Delta_{+} - (\g -2i)^{2} \right)\\
    \Pi_{j}^{\g} &= (-1)^{\lfloor \g \rfloor} \prod_{i=0}^{j-1} \left( \tilde \Delta_{+} - (\g -2i)^{2} \right) \prod_{i=j+1}^{\lfloor \g \rfloor}\left( \tilde\Delta_{+} - (\g -2i)^{2} \right)\\
    &= (-1)^{\lfloor \g \rfloor} \prod_{i=0}^{j-1}(\tilde\Delta_{+} - (\g - 2i)^{2}) \prod_{i=0}^{\lfloor \g \rfloor - j - 1}(\tilde\Delta_{+} - (\g - 2j - 2 - 2i)^{2}).
  \end{align*}
  We will also sometimes use notation such as
  \begin{align*}
    \Pi_{j}^{\g} &= \frac{\Pi^{\g}}{-(\tilde \Delta_{+} - (\g-2j)^{2})}
  \end{align*}
  to indicate the differential operator obtained by removing the factor $-(\tilde\Delta_{+} -(\g-2j)^{2})$ from the expression used to define $\Pi^{\g}$.
  Using \eqref{eq:uj-satisfies-scattering-equation} and the Green-type identity (Lemma \ref{lem:green-identity}), we get
  \begin{align*}
    &\int_{X} \rho^{\frac{n}{2} - \g} U_{j} L_{2k}^{+}(\rho^{\frac{n}{2}-\g}(V-\tilde V)) dvol_{g} \\
    &= \int_{X} \rho^{\frac{n}{2}-\g} U_{j} L_{2k}^{+} (\rho^{\frac{n}{2}-\g}(V - \tilde V)) dvol_{g} \\
    &+ \int_{X} \left( \tilde\Delta_{+} - (\g-2j)^{2} \right)(\rho^{\frac{n}{2}-\g}U_{j}) \Pi^{\g}_{j}(\rho^{\frac{n}{2}-\g}(V-\tilde V)) dvol_{g}\\
    &= -\int_{X} \rho^{\frac{n}{2}-\g} U_{j} \Delta_{+} \Pi^{\g}_{j}(\rho^{\frac{n}{2}-\g}(V-\tilde V)) dvol_{g} \\
    &+ \int_{X} \Delta_{+}(\rho^{\frac{n}{2}-\g}U_{j}) \Pi^{\g}_{j}(\rho^{\frac{n}{2}-\g}(V-\tilde V)) dvol_{g}\\
    &=-\int_{M}\rho^{1-n}\p_{\rho}(\rho^{\frac{n}{2}-\g}U_{j}) \Pi^{\g}_{j}(\rho^{\frac{n}{2}-\g}(V-\tilde V))\big|_{\rho=0} dvol_{g} \\
    &+ \int_{M} \rho^{1-n}\rho^{\frac{n}{2}-\g}U_{j} \p_{\rho}\Pi^{\g}_{j}(\rho^{\frac{n}{2}-\g}(V-\tilde V))\big|_{\rho=0} dvol_{g}.
  \end{align*}
  Therefore,
  \begin{align*}
    A_{2} &= \int_{X} \rho^{\frac{n}{2}-\g}\tilde U L_{2k}^{+} ( \rho^{\frac{n}{2} - \g} (V - \tilde V))dvol_{g}\\
    &= \sum_{j=0}^{\lfloor \g \rfloor} \int_{X} \rho^{\frac{n}{2} - \g} U_{j} L_{2k}^{+}(\rho^{\frac{n}{2}-\g}(V-\tilde V)) dvol_{g} \\
    &=-\sum_{j = 0 }^{\lfloor \g \rfloor} \int_{M}\rho^{1-n}\p_{\rho}(\rho^{\frac{n}{2}-\g}U_{j}) \Pi^{\g}_{j}(\rho^{\frac{n}{2}-\g}(V-\tilde V))\big|_{\rho=0} dvol_{g} \\
    &+ \sum_{j=0}^{\lfloor \g \rfloor} \int_{M} \rho^{1-n}\rho^{\frac{n}{2}-\g}U_{j} \p_{\rho}\Pi^{\g}_{j}(\rho^{\frac{n}{2}-\g}(V-\tilde V))\big|_{\rho=0} dvol_{g},
  \end{align*}
  thereby completing the goal of expressing $A_{2}$ as boundary integrals.
  The goal is to now express the integrands in terms of boundary operators acting on $U$ and $V$.
  Now, by \eqref{eq:scattering-uj}, we have for $j = 0, \ldots, \lfloor \g/2\rfloor$ that
  \begin{align*}
    \rho^{1-n} \rho^{\frac{n}{2}-\g} U_{j} &= \rho^{-\frac{n}{2} - \g + 2j + 1}F_{j} + O(\rho^{-\frac{n}{2} + \g  -2j + 1 })\\
    \rho^{1-n} \p_{\rho} (\rho^{\frac{n}{2}-\g}U_{j}) &= (\frac{n}{2}-\g+2j)\rho^{-\frac{n}{2} - \g + 2j} F_{j} + O(\rho^{-\frac{n}{2} + \g - 2j})
  \end{align*}
  and for $j = \lfloor \g /2 \rfloor + 1, \ldots, \lfloor \g \rfloor$, that
  \begin{align*}
    \rho^{1-n} \rho^{\frac{n}{2}-\g} U_{j} &= \rho^{-\frac{n}{2} + \g - 2j + 1}F_{j} + O(\rho^{-\frac{n}{2} -\g+2j+1 })\\
    \rho^{1-n} \p_{\rho} (\rho^{\frac{n}{2}-\g}U_{j}) &= (\frac{n}{2}+\g-2j)\rho^{-\frac{n}{2} + \g - 2j} F_{j} + O(\rho^{-\frac{n}{2} - \g + 2j})
  \end{align*}
  Since we know what $F_{j}|_{\rho=0}$ evaluate to by \eqref{eq:fj-gj-boundary-values}, it is clear it is enough to compute
  \begin{align*}
    &\rho^{-\frac{n}{2} - \g + 2j} \Pi^{\g}_{j}(\rho^{\frac{n}{2}-\g}(V - \tilde V)) \quad \text{ and } \quad \rho^{-\frac{n}{2} - \g + 2j + 1} \p_{\rho}\Pi^{\g}_{j}(\rho^{\frac{n}{2}-\g}(V-\tilde V))
  \end{align*}
  for $j = 0, \ldots, \lfloor \g/2 \rfloor$ and
  \begin{align*}
    &\rho^{-\frac{n}{2} + \g - 2j} \Pi^{\g}_{j}(\rho^{\frac{n}{2}-\g}(V - \tilde V)) \quad \text{ and } \quad \rho^{-\frac{n}{2} + \g - 2j + 1} \p_{\rho}\Pi^{\g}_{j}(\rho^{\frac{n}{2}-\g}(V-\tilde V))
  \end{align*}
  for $j = \lfloor \g/2 \rfloor + 1,\ldots, \lfloor \g \rfloor$.
  (We will see that these evaluations are finite and generally nonzero.)

  For brevity, let $f = V - \tilde V$.
  Given the boundary conditions ($B_{0}^{2\g} = B_{2}^{2\g} = \cdots = 0$ etc), Corollary \ref{cor:general-series-expansion-in-terms-of-B} allows us to write
  \[
    f = \sum_{\ell=0}^{\oo} \rho^{\lfloor \g \rfloor + 2 + 2\ell} f _{\lfloor \g \rfloor + 2 + 2\ell} + \sum_{\ell=0}^{\oo} \rho^{\lfloor \g \rfloor + 2\ell + 2[\g]} f_{\lfloor \g \rfloor + 2\ell + 2[\g]},
  \]
  for some boundary functions $f_{\a} \in C^{\oo}(M)$, where
  \begin{align*}
    f_{\lfloor \g \rfloor + 2\ell } &= B_{\lfloor \g \rfloor + 2\ell }^{2\g}\left( f - \sum_{m=0}^{\ell-1} \rho^{\lfloor \g \rfloor + 2m } f_{\lfloor \g \rfloor + 2m } \right)\\
    f_{\lfloor \g \rfloor + 2\ell + 2[\g]} &= B_{\lfloor \g \rfloor + 2\ell + 2[\g]}^{2\g}\left( f - \sum_{m=0}^{\ell-1} \rho^{\lfloor \g \rfloor + 2m + 2[\g]} f_{\lfloor \g \rfloor + 2m + 2[\g]} \right).
  \end{align*}

  It will be useful to understand how $\Pi_{j}^{\g}$ and $\Pi^{\g}$ act on $\rho^{\frac{n}{2}+1-\g}f$.
  Note
  \begin{align*}
    \Pi^{\g} &= \prod_{m=0}^{\lfloor \g/2 \rfloor}(-\tilde\Delta_{+} + (\g-2m)^{2}) \prod_{m=0}^{\lfloor \g/2 \rfloor  - 1} (-\tilde\Delta_{+} + (\g - \lfloor \g \rfloor - 2 -2m)^{2}).
  \end{align*}
  Using Lemma \ref{lm1.1}, we make the following observations for $0 \leq \ell \leq \lfloor \g /2 \rfloor$: there holds
  \begin{align*}
    &\Pi^{\g} (\rho^{\frac{n}{2}-\g} \rho^{\lfloor \g \rfloor + 2 + 2\ell } f_{\lfloor \g \rfloor + 2 + 2\ell}) \in \rho^{\frac{n}{2}-\g+2\lfloor \g \rfloor + 2} \ceven\\
    &\Pi^{\g} (\rho^{\frac{n}{2}-\g} \rho^{\lfloor \g \rfloor + 2\ell + 2[\g] } f_{\lfloor \g \rfloor + 2\ell + 2[\g]}) \in \rho^{\frac{n}{2}+\g+2} \ceven
  \end{align*}
  and so the leading order power of $\rho$ in $\Pi^{\g}(\rho^{\frac{n}{2}-\g}f)$ is $\rho^{\frac{n}{2}-\g + 2\lfloor \g \rfloor + 2}$;
  for $0 \leq j \leq \lfloor \g/2\rfloor $, there holds
  \begin{align*}
    &\Pi^{\g}_{j} (\rho^{\frac{n}{2}-\g} \rho^{\lfloor \g \rfloor + 2 + 2\ell } f_{\lfloor \g \rfloor + 2 + 2\ell}) \in  \rho^{\frac{n}{2} - \g + 2\lfloor \g \rfloor + 2} \ceven\\
    &\Pi^{\g}_{j} (\rho^{\frac{n}{2}-\g} \rho^{\lfloor \g \rfloor + 2\ell + 2[\g] } f_{\lfloor \g \rfloor + 2\ell + 2[\g]}) \in \rho^{\frac{n}{2}+\g - 2j} \ceven
  \end{align*}
  and so the leading order power of $\rho$ in $\Pi^{\g}_{j}(\rho^{\frac{n}{2}-\g}f)$ is $\rho^{\frac{n}{2}+\g-2j}$;
  for $\lfloor \g/2 \rfloor + 1 \leq j \leq \lfloor \g\rfloor $, there holds
  \begin{align*}
    &\Pi^{\g}_{j} (\rho^{\frac{n}{2}-\g} \rho^{\lfloor \g \rfloor + 2 + 2\ell } f_{\lfloor \g \rfloor + 2 + 2\ell}) \in \rho^{\frac{n}{2}-\g+2j} \ceven\\
    &\Pi^{\g}_{j} (\rho^{\frac{n}{2}-\g} \rho^{\lfloor \g \rfloor + 2\ell + 2[\g] } f_{\lfloor \g \rfloor + 2\ell + 2[\g]}) \in \rho^{\frac{n}{2}+\g + 2} \ceven
  \end{align*}
  and so the leading order power of $\rho$ in $\Pi^{\g}_{j}(\rho^{\frac{n}{2}-\g}f)$ is $\rho^{\frac{n}{2}-\g + 2j}$.
  From these properties it is clear that there are many interactions between the various terms in $\Pi_{j}^{\g} (\rho^{\frac{n}{2}-\g} (V - \tilde V))$.
  The main difficulty is showing that most of these terms in fact cancel.
  This is what we will show now.
  To do so, we consider two cases separately.

  Before we begin, we recall that the goal is to compute
  \[
    \rho^{-\frac{n}{2} \mp \g \pm 2j} \Pi_{j}^{\g} (\rho^{\frac{n}{2}-\g}f) |_{\rho=0} \text{ and } \rho^{-\frac{n}{2} \mp \g \pm 2j+1} \p_{\rho} \Pi_{j}^{\g} (\rho^{\frac{n}{2}-\g}f) |_{\rho=0}.
  \]
  In fact, we show
  \begin{align*}
    &\rho^{-\frac{n}{2} \mp \g \pm 2j} \Pi_{j}^{\g} (\rho^{\frac{n}{2}-\g}f) |_{\rho=0}\\
    &=
    \begin{cases}
      \pi_{j}^{\g} B_{2\lfloor \g \rfloor - 2j + 2[\g]}^{2\g}f, & 0 \leq j \leq \lfloor \g/2 \rfloor, (-,+)\\
      \pi_{j}^{\g} B_{2j}^{2\g}f, & \lfloor \g/2 \rfloor  + 1 \leq j \leq \lfloor \g \rfloor, (+,-)
    \end{cases}\\
    &\rho^{-\frac{n}{2} \mp \g \pm 2j+1} \p_{\rho} \Pi_{j}^{\g} (\rho^{\frac{n}{2}-\g}f) |_{\rho=0} \\
    &=
    \begin{cases}
      (\frac{n}{2} +\g-2j)\pi_{j}^{\g} B_{2\lfloor \g \rfloor - 2j + 2[\g]}^{2\g}f, & 0 \leq j \leq \lfloor \g/2 \rfloor, (-,+)\\
      (\frac{n}{2}-\g+2j)\pi_{j}^{\g} B_{2j}^{2\g}f, & \lfloor \g/2 \rfloor + 1 \leq j \leq \lfloor \g \rfloor, (+,-)
    \end{cases}
  \end{align*}
  for a nonzero constant $\pi_{j}^{\g}$ given by
  \begin{align*}
    \pi_{j}^{\g} &= \rho^{-\frac{n}{2} - \g + 2j}\Pi_{j}^{\g}(\rho^{\frac{n}{2} + \g - 2j})\\
    &= \rho^{-\frac{n}{2} + \g - 2j} \Pi_{j}^{\g}(\rho^{\frac{n}{2} - \g + 2j})\\
    &= \prod_{m=0}^{j-1} \left(-(\g-2j)^{2} + (\g-2m)^{2} \right) \prod_{m=0}^{\lfloor \g \rfloor - j -1} \left(-(\g-2j)^{2} + (\g-2j-2-2m)^{2} \right).
  \end{align*}

  \medskip
  \noindent
  \textbf{Case 1.}: $0 \leq j \leq \lfloor \g /2\rfloor$\\
  \noindent
  It is easy to see that there is only one term to consider in case $j = \lfloor \g /2 \rfloor$ and thus there is no need to consider interacting terms in this case.
  We therefore only need handle interaction of terms in the case $0 \leq j < \lfloor \g /2 \rfloor$.
  Let $0 \leq 2\ell \leq  \lfloor \g \rfloor - 2j - 2$.
  Note $\rho^{\frac{n}{2}-\g}\rho^{\lfloor \g \rfloor + 2\ell + 2[\g]} = \rho^{\frac{n}{2}+1+\g - \lfloor \g \rfloor + 2\ell}$.
  Note that nonzero elements in $\rho^{\frac{n}{2}+\g - 2j}\ceven$ are eigenfunctions$\mod \rho^{\frac{n}{2} + 2 +\g-2j}\ceven$ of $\Pi_{j}^{\g}$ and that
  \[
    \Pi_{j}^{\g}: \rho^{\frac{n}{2}+1+\g - \lfloor \g \rfloor + 2\ell}\ceven \to \rho^{\frac{n}{2}+1+\g - 2j}\ceven.
  \]
  It follows that there is interaction between the two terms:
  \begin{align*}
    &\rho^{-\frac{n}{2} - \g + 2j} \Pi_{j}^{\g} \left( \rho^{\frac{n}{2}+1+\g-\lfloor \g \rfloor +2\ell} f_{\lfloor \g \rfloor + 2\ell + 2[\g]} \right)\big|_{\rho=0}
  \end{align*}
  and
  \begin{align*}
    -& \rho^{-\frac{n}{2} - \g + 2j}\Pi_{j}^{\g} \left( \rho^{\frac{n}{2}+1+\g-2j} B_{2\g-2j}^{2\g}\left( \rho^{\lfloor \g \rfloor + 2\ell + 2[\g]} f_{\lfloor \g \rfloor + 2\ell + 2[\g]} \right) \right)\big|_{\rho=0}.
  \end{align*}
  We show that these terms cancel, thereby showing that the only possible nonzero term in $\rho^{-\frac{n}{2} - \g + 2j} \Pi_{j}^{\g}(\rho^{\frac{n}{2}-\g}f)$ is
  \[
    \rho^{-\frac{n}{2} - \g + 2j} \Pi_{j}^{\g} \left( \rho^{\frac{n}{2}+\g-2j}  B_{2\g-2j}^{2\g} f \right)|_{\rho=0}.
  \]

  Fix $0 \leq 2\ell \leq \lfloor \g \rfloor - 2j - 2$.
  Let
  \begin{align*}
    \Pi_{j,0}^{\g} &= \left( -\tilde\Delta_{+} + (\g - 2j - 2)^{2} \right) \cdots \left( -\tilde\Delta_{+} + (\g - \lfloor \g \rfloor + 2\ell )^{2} \right)\\
    \Pi_{j,1}^{\g} &= \frac{\Pi_{j}^{\g}}{\Pi_{j,0}^{\g}}\\
    &= \left( -\tilde \Delta_{+} + \g^{2} \right) \cdots \left( -\tilde\Delta_{+} + (\g - 2j + 2)^{2} \right) \left(- \tilde\Delta_{+} + (\g - \lfloor \g \rfloor + 2\ell - 2)^{2} \right) \cdots \left(- \tilde\Delta_{+} + (\g - 2\lfloor \g \rfloor)^{2} \right)\\
    &= \prod_{m = 0 }^{j-1} \left( -\tilde\Delta_{+} + (\g - 2m)^{2} \right) \prod_{m=0}^{\lfloor \g/2 \rfloor + \ell - 1} \left(-\tilde\Delta_{+} + (\g - \lfloor \g \rfloor + 2\ell - 2 - 2m)^{2} \right).
  \end{align*}
  Note that elements in $\rho^{\frac{n}{2}+\g-2j} \ceven$ are eigenfunctions $\mod \rho^{\frac{n}{2} + 2 +\g-2j} \ceven$ of $\Pi_{j,1}^{\g}$ and
  \[
    \Pi_{j,0}^{\g} : \rho^{\frac{n}{2}+\g-\lfloor \g \rfloor + 2 \ell} \ceven \to \rho^{\frac{n}{2}+\g-2j} \ceven.
  \]
  Thus there is a constant $\pi_{j,1}^{\g}$ such that
  \begin{align*}
    \Pi_{j}^{\g} \left( \rho^{\frac{n}{2}+\g-\lfloor \g \rfloor + 2\ell} f_{\lfloor \g \rfloor + 2\ell + 2[\g]} \right) &= \pi_{j,1}^{\g} \Pi_{j,0}^{\g} \left( \rho^{\frac{n}{2}+\g-\lfloor \g \rfloor + 2\ell} f_{\lfloor \g \rfloor + 2\ell + 2[\g]} \right) \mod \rho^{\frac{n}{2} + 2 +\g-2j}\ceven.
  \end{align*}
  We see $\pi_{j,1}^{\g}$ is given by
  \[
    \pi_{j,1}^{\g} = \rho^{-\frac{n}{2} - \g + 2j} \Pi_{j,1}^{\g} \rho^{\frac{n}{2}+\g-2j}.
  \]
  Using \eqref{eq:laplace-on-rho}, we then have
  \[
    \pi_{j,1}^{\g} = \prod_{m=0}^{j-1}\left(- (\g-2j)^{2} + (\g-2m)^{2} \right)\prod_{m=0}^{\lfloor \g/2\rfloor + \ell - 1} \left(-(\g-2j)^{2} + (\g - \lfloor \g \rfloor + 2\ell - 2 - 2m)^{2} \right).
  \]
  We record:
  \begin{align*}
    &\rho^{-\frac{n}{2} - \g + 2j} \Pi_{j}^{\g} \left( \rho^{\frac{n}{2}+\g-\lfloor \g \rfloor +2\ell} f_{\lfloor \g \rfloor + 2\ell + 2[\g]} \right)\big|_{\rho=0} = \pi_{j,1}^{\g} \rho^{-\frac{n}{2}-\g+2j}\Pi_{j,0}^{\g}(\rho^{\frac{n}{2}+\g-\lfloor \g \rfloor + 2\ell}f_{\lfloor \g \rfloor + 2\ell + 2[\g]})\big|_{\rho=0}.
  \end{align*}

  We deal with the other term now.
  Note
  \begin{align*}
    B_{2\g-2j}^{2\g,\e} &= \rho^{-\frac{n}{2} - \g + 2j} \circ \left( \tilde\Delta_{+} - (\g-2j)^{2} + \e \right) \prod_{m=0,m\neq j}^{\lfloor \g \rfloor - j}\left( \tilde\Delta_{+} - (\g-2m)^{2} \right)\\
    &\circ \prod_{m = j+1}^{\lfloor \g \rfloor} \left( \tilde\Delta_{+} - (\g-2m)^{2} \right) \circ \rho^{\frac{n}{2}-\g}|_{\rho=0}\\
    &= \rho^{-\frac{n}{2} - \g + 2j} \Pi_{j,3}^{\g,\e} \Pi_{j,2}^{\g} \Pi_{j,0}^{\g} \circ \rho^{\frac{n}{2}-\g}|_{\rho=0},
  \end{align*}
  where
  \begin{align*}
    \Pi_{j,3}^{\g,\e} &= \left( \tilde\Delta_{+} - (\g-2j)^{2} + \e \right) \prod_{m=0,m\neq j}^{\lfloor \g \rfloor - j}\left( \tilde\Delta_{+} - (\g-2m)^{2} \right)\\
    \Pi_{j,2}^{\g} &= \prod_{m=j+1}^{\lfloor \g \rfloor} \left( \tilde\Delta_{+} - (\g-2m)^{2} \right) / \Pi_{j,0}^{\g}\\
    &=(-1)^{\lfloor \g/2\rfloor - \ell -j} \prod_{m=0}^{\lfloor \g/2\rfloor + \ell - 1} \left( \tilde\Delta_{+} - (\g - \lfloor \g \rfloor + 2\ell - 2 -2m)^{2} \right).
  \end{align*}

  Since
  \[
    \Pi_{j,0}^{\g}: \rho^{\frac{n}{2}+\g-\lfloor \g \rfloor + 2\ell} \ceven \to \rho^{\frac{n}{2}+\g-2j} \ceven
  \]
  and elements of $\rho^{\frac{n}{2}+\g-2j} \ceven$ are eigenfunctions $\mod \rho^{\frac{n}{2} + 2 +\g-2j} \ceven$ of $\Pi_{j,2}^{\g}$ and $\Pi_{j,3}^{\g,\e}$, it is easy to see that there are constants $\pi_{j,3}^{\g,\e}, \pi_{j,2}^{\g}$ such that
  \begin{align*}
    &\Pi_{j,3}^{\g,\e} \Pi_{j,2}^{\g} \Pi_{j,0}^{\g}\rho^{\frac{n}{2}-\g}\rho^{\lfloor \g \rfloor + 2\ell + 2[\g]}f_{\lfloor \g \rfloor + 2\ell + 2[\g]} \\
    &= \Pi_{j,3}^{\g,\e} \Pi_{j,2}^{\g} \Pi_{j,0}^{\g}\rho^{\frac{n}{2}+\g-\lfloor \g \rfloor + 2\ell} f_{\lfloor \g \rfloor + 2\ell + 2[\g]} \\
    &= \pi_{j,3}^{\g,\e} \pi_{j,2}^{\g} \Pi_{j,0}^{\g} \rho^{\frac{n}{2}+\g-\lfloor \g \rfloor + 2\ell} f_{\lfloor \g \rfloor + 2\ell + 2[\g]} \mod \rho^{\frac{n}{2} + 2+\g-2j}\ceven.
  \end{align*}
  Using \eqref{eq:laplace-on-rho}, we compute
  \begin{align*}
    \pi_{j,2}^{\g } &= (-1)^{\lfloor \g/2 \rfloor - \ell - j} \prod_{m=0}^{\lfloor \g /2 \rfloor + \ell - j} \left( (\g- 2j)^{2} - (\g - \lfloor \g \rfloor + 2 \ell - 2 - 2m)^{2} \right)\\
    \pi_{j,3}^{\g , \e} &= \e \prod_{m=0,m\neq j}^{\lfloor \g \rfloor - j} \left( (\g-2j)^{2} - (\g-2m)^{2} \right).
  \end{align*}
  We lastly record
  \begin{align*}
    \tilde b_{2\g-2j}^{\e} &=  \e\prod_{m=0,m \neq j}^{\lfloor \g \rfloor - j} \left( (\g-2j)^{2} - ( \g - 2m)^{2} \right)   \prod_{m=0}^{\lfloor \g \rfloor - j - 1} \left( (\g-2j)^{2} - (\g  + 2m - 2\lfloor \g \rfloor)^{2} \right).
  \end{align*}
  At last, it is easy to compute
  \begin{align*}
    \lim_{\e \to 0}\frac{\pi_{j,3}^{\g , \e} \pi_{j,2}^{\g} \pi_{j}^{\g}}{\tilde b_{2\g - 2j}^{\e} \pi_{j,1}^{\g}} &= \frac{\prod_{m=0,m\neq j}^{\lfloor \g \rfloor - j}\left( (\g-2j)^{2} - (\g  - 2m)^{2} \right)}{\prod_{m=0,m\neq j}^{\lfloor \g \rfloor - j}\left( (\g-2j)^{2} - (\g-2m)^{2} \right)} \\
    &\times \frac{\prod_{m=0}^{\lfloor \g/2 \rfloor - j - 1}\left( (\g-2j)^{2} - (\g-\lfloor \g \rfloor + 2\ell - 2 - 2m)^{2} \right)}{\prod_{m=0}^{\lfloor \g \rfloor - j - 1}\left( \g-2j \right)^{2} - (\g+2m-2\lfloor \g \rfloor)^{2}}\\
    &\times \frac{\prod_{m=0}^{j-1} \left( (\g-2j)^{2} - (\g-2m)^{2} \right)}{\prod_{m=0}^{j-1} \left( (\g-2j)^{2} - (\g-2m)^{2} \right)} \\
    &\times \frac{\prod_{m=0}^{\lfloor \g \rfloor - j - 1}\left( (\g-2j)^{2} - (\g-2j-2-2m)^{2} \right)}{\prod_{m=0}^{\lfloor \g/ 2 \rfloor + \ell - 1}\left( (\g-2j)^{2} - (\g-\lfloor \g \rfloor + 2\ell - 2 - 2m)^{2} \right)}\\
    &=1.
  \end{align*}
  This is enough to conclude
  \begin{align*}
    &\rho^{-\frac{n}{2} - \g + 2j} \Pi_{j}^{\g} \left( \rho^{\frac{n}{2}+\g-\lfloor \g \rfloor +2\ell} f_{\lfloor \g \rfloor + 2\ell + 2[\g]} \right)\big|_{\rho=0}\\
    -& \rho^{-\frac{n}{2} - \g + 2j}\Pi_{j}^{\g} \left( \rho^{\frac{n}{2}+\g-2j} B_{2\lfloor \g \rfloor - 2j + 2[\g]}^{2\g}\left( \rho^{\lfloor \g \rfloor + 2\ell + 2[\g]} f_{\lfloor \g \rfloor + 2\ell + 2[\g]} \right) \right)\big|_{\rho=0}\\
    &=0
  \end{align*}
  and therefore
  \begin{align*}
    \rho^{-\frac{n}{2} - \g + 2j} \Pi_{j}^{\g} \rho^{\frac{n}{2}-\g}f|_{\rho=0}& = \pi_{j}^{\g} B_{2\g-2j}^{2\g}f\\
    \rho^{-\frac{n}{2} - \g + 2j + 1}\p_{\rho} \Pi_{j}^{\g} \rho^{\frac{n}{2}-\g}f |_{\rho=0} &= (\frac{n}{2}+\g-2j)\pi_{j}^{\g} B_{2\g-2j}^{2\g}f,
  \end{align*}
  which is what we wanted to show.
  Then
  \begin{align*}
    \rho^{1-n}\p_{\rho}(\rho^{\frac{n}{2}-\g}U_{j}) \Pi_{j}(\rho^{\frac{n}{2}-\g}(V-\tilde V))|_{\rho=0}&= \pi_{j}^{\g} (\frac{n}{2}-\g+2j) B_{2j}^{2\g}U B_{2\g-2j}^{2\g}(V - \tilde V)\\
    \rho^{1-n}(\rho^{\frac{n}{2}-\g}U_{j}) \p_{\rho} \Pi_{j}(\rho^{\frac{n}{2}-\g}(V-\tilde V))|_{\rho=0} &= \pi_{j}^{\g} (\frac{n}{2} + \g - 2j) B_{2j}^{2\g}U B_{2\g-2j}^{2\g}(V - \tilde V).
  \end{align*}
  Consequently,
  \begin{align*}
    &-\sum_{j = 0 }^{\lfloor \g/2 \rfloor} \int_{M}\rho^{1-n}\p_{\rho}(\rho^{\frac{n}{2}-\g}U_{j}) \Pi_{j}(\rho^{\frac{n}{2}-\g}(V-\tilde V))\big|_{\rho=0} dvol_{g} \\
    &+ \sum_{j=0}^{\lfloor \g/2 \rfloor} \int_{M} \rho^{1-n}\rho^{\frac{n}{2}-\g}U_{j} \p_{\rho}\Pi_{j}(\rho^{\frac{n}{2}-\g}(V-\tilde V))\big|_{\rho=0} dvol_{g}\\
    &=\sum_{j=0}^{\lfloor \g/2 \rfloor} \int_{M}  \pi_{j}^{\g}(2\g-4j) B_{2j}^{2\g}U  B_{2\g-2j}^{2\g}(V - \tilde V) dvol_{g}.
  \end{align*}
  We have by Theorem \ref{thm:boundary-to-fractional-operator} that
  \[
    B^{2\g}_{2\g-2j}\tilde V = c_{\g,j} P_{\g-2j} B_{2j}^{2\g}V
  \]
  and so
  \begin{align*}
    &\sum_{j=0}^{\lfloor \g/2 \rfloor} \int_{M}  \pi_{j}^{\g}(2\g-4j) B_{2j}^{2\g}U  B_{2\g-2j}^{2\g}(V - \tilde V) dvol_{g}\\
    &=\sum_{j=0}^{\lfloor \g/2 \rfloor} \pi_{j}^{\g}(2\g-4j)\int_{M}   B_{2j}^{2\g}U  B_{2\g-2j}^{2\g}V  dvol_{g}\\
    &- \sum_{j=0}^{\lfloor \g/2 \rfloor} c_{\g,j}\pi_{j}^{\g}(2\g-4j)\int_{M}   B_{2j}^{2\g}U P_{\g-2j}B_{2j}^{2\g} V    dvol_{g}.
  \end{align*}

  \medskip
  \noindent
  \textbf{Case 2.}: $\lfloor \g/2 \rfloor + 1 \leq j \leq  \lfloor \g \rfloor$\\
  \noindent
  It is easy to see that there is only one term to consider in case $j = \lfloor \g /2 \rfloor + 1$ and thus there is no need to consider interacting terms in this case.
  We therefore only need handle interaction of terms in the case $\lfloor \g/2 \rfloor + 1 < j \leq  \lfloor \g \rfloor$.
  Let $0 \leq 2 \ell \leq 2j - \lfloor \g \rfloor -4$.
  Note that elements in $\rho^{\frac{n}{2}-\g + 2j} \ceven$ are eigenfunctions $\mod \rho^{\frac{n}{2} + 2 -\g+2j}\ceven$ of $\Pi_{j}^{\g}$ and that
  \[
    \Pi_{j}^{\g} : \rho^{\frac{n}{2}-\g+\lfloor \g \rfloor + 2 + 2\ell} \ceven \to \rho^{\frac{n}{2}-\g + 2j}\ceven.
  \]
  It follows that there is interaction between the two terms
  \[
    \rho^{-\frac{n}{2} + \g - 2j}\Pi_{j}^{\g}(\rho^{\frac{n}{2}-\g + \lfloor \g \rfloor + 2 + 2\ell} f_{\lfloor \g \rfloor + 2 + 2 \ell})|_{\rho=0}
  \]
  and
  \[
    - \rho^{-\frac{n}{2} + \g - 2j} \Pi_{j}^{\g} \left( \rho^{\frac{n}{2}-\g+2j} B_{2j}^{2\g} \left(  \rho^{\lfloor \g \rfloor + 2 + 2 \ell} f_{\lfloor \g \rfloor + 2 + 2\ell} \right)\right)|_{\rho=0}.
  \]
  We show that these terms cancel, thereby showing that the only possibly nonzero term in $\rho^{-\frac{n}{2} + \g - 2j} \Pi_{j} (\rho^{\frac{n}{2}-\g}f)$ is
  \[
    \rho^{-\frac{n}{2} + \g - 2j} \Pi_{j} (\rho^{\frac{n}{2}-\g + 2j} B_{2j}^{2\g}f).
  \]

  Fix $0 \leq 2 \ell \leq 2j - \lfloor \g \rfloor - 4$.
  Let
  \begin{align*}
    \Pi_{j,0}^{\g} &= \left(-\tilde\Delta_{+} + (\g - \lfloor \g \rfloor - 2 -2 \ell)^{2} \right) \cdots  \left(-\tilde\Delta_{+} + (\g - 2j + 2)^{2} \right)\\
    \Pi_{j,1}^{\g} &= \frac{\Pi_{j}^{\g}}{\Pi_{j,0}^{\g}}\\
    &= \left(-\tilde\Delta_{+} + \g^{2} \right) \cdots \left( -\tilde\Delta_{+} + (\g - \lfloor \g \rfloor - 2\ell)^{2} \right) \left(-\tilde\Delta_{+} + (\g-2j-2)^{2} \right) \cdots \left(-\tilde\Delta_{+} + (\g-2\lfloor \g \rfloor)^{2} \right)\\
    &= \prod_{m=0}^{\lfloor\g/2\rfloor + \ell} \left(-\tilde\Delta_{+} + (\g - 2m)^{2} \right) \prod_{m= 0 }^{\lfloor \g \rfloor - j - 1} \left(-\tilde\Delta_{+} + (\g -2j -2 - 2m  )^{2} \right).
  \end{align*}
  Note that elements in $\rho^{\frac{n}{2}-\g + 2j} \ceven$ are eigenfunctions $\mod \rho^{\frac{n}{2} + 2-\g+2j}\ceven$ of $\Pi_{j,1}^{\g}$ and that
  \[
    \Pi_{j,0}^{\g} : \rho^{\frac{n}{2}-\g+\lfloor \g \rfloor + 2 + 2\ell} \ceven \to \rho^{\frac{n}{2}-\g + 2j}\ceven.
  \]
  Then there is a constant $\pi_{j,1}^{\g}$ such that
  \[
    \Pi_{j}^{\g}(\rho^{\frac{n}{2}-\g + \lfloor \g \rfloor + 2 + 2\ell} f_{\lfloor \g \rfloor + 2 + 2\ell}) =  \pi_{j,1}^{\g} \Pi_{j,0}^{\g}(\rho^{\frac{n}{2}-\g+\lfloor \g \rfloor + 2 + 2\ell} f_{\lfloor \g \rfloor + 2 + 2\ell}) \mod \rho^{\frac{n}{2} + 2 -\g+2j} \ceven.
  \]
  Using \eqref{eq:laplace-on-rho}, we see that $\pi_{j,1}^{\g}$ is given by
  \begin{align*}
    \pi_{j,1}^{\g} &= \rho^{-\frac{n}{2} + \g - 2j} \Pi_{j,1}^{\g}\rho^{\frac{n}{2}-\g+2j}\\
    &=\prod_{m=0}^{\lfloor \g /2 \rfloor + \ell}\left(-(\g-2j)^{2} + (\g-2m)^{2}  \right)\prod_{m=0}^{\lfloor \g \rfloor - j - 1} \left(-(\g-2j)^{2} + (\g-2j-2-2m)^{2} \right).
  \end{align*}
  We record
  \begin{align*}
    \rho^{-\frac{n}{2}+\g-2j} \Pi_{j}^{\g} \left( \rho^{\frac{n}{2}-\g+\lfloor \g \rfloor + 2 + 2\ell} f_{\lfloor \g \rfloor + 2 + 2\ell} \right)|_{\rho=0} = \pi_{j,1}^{\g}\rho^{-\frac{n}{2} + \g - 2j} \Pi_{j,0}^{\g} \left( \rho^{\frac{n}{2}-\g+\lfloor \g \rfloor + 2 + 2\ell} f_{\lfloor \g \rfloor + 2 + 2\ell} \right)|_{\rho=0}.
  \end{align*}

  We deal with the other term now.
  Note
  \begin{align*}
    B_{2j}^{2\g,\e}  &= \rho^{-\frac{n}{2} - \g + 2j} \circ \left( \tilde\Delta_{+} - (\g-2j)^{2} + \e \right) \circ \prod_{m=0}^{j-1} \left( \tilde\Delta_{+} - (\g-2m)^{2} \right)\\
    &\circ \prod_{m=\lfloor \g \rfloor - j-1, m \neq j}^{\lfloor \g \rfloor} \left( \tilde \Delta_{+} - (\g-2m)^{2} \right) \circ \rho^{\frac{n}{2}-\g} |_{\rho=0}\\
    &=  \rho^{-\frac{n}{2} + \g - 2j} \circ \Pi_{j,3}^{\g,\e}\Pi_{j,2}^{\g}\Pi_{j,0}^{\g} \circ \rho^{\frac{n}{2}-\g}|_{\rho=0}
  \end{align*}
  where
  \begin{align*}
    \Pi_{j,3}^{\g,\e} &= \left( \tilde\Delta_{+} - (\g-2j)^{2} + \e \right)\prod_{m=\lfloor \g \rfloor - j-1, m \neq j}^{\lfloor \g \rfloor} \left( \tilde\Delta_{+} - (\g -2m  )^{2} \right)\\
    \Pi_{j,2}^{\g} &= \prod_{m=0}^{j-1}(\tilde\Delta_{+} - (\g  -2m)^{2}) / \Pi_{j,0}^{\g}\\
    &=(-1)^{j-\lfloor \g/2 \rfloor - 1 - \ell}\prod_{m=0}^{\lfloor \g /2 \rfloor + \ell}(\tilde\Delta_{+} - (\g  -2m)^{2}).
  \end{align*}
  Since
  \[
    \Pi_{j,0}^{\g } : \rho^{\frac{n}{2}-\g + \lfloor \g \rfloor + 2 + 2\ell} \ceven \to \rho^{\frac{n}{2}-\g+2j}\ceven
  \]
  and elements of $\rho^{\frac{n}{2}-\g+2j}\ceven$ are eigenfunctions $\mod \rho^{\frac{n}{2} + 2 +\g-2j}\ceven$ of $\Pi_{j,2}^{\g}$ and $\Pi_{j,3}^{\g,\e}$, it is easy to see that there are constants $\pi_{j,3}^{\g,\e},\pi_{j,2}^{\g}$ such that
  \begin{align*}
    &\Pi_{j,3}^{\g,\e}\Pi_{j,2}^{\g}\Pi_{j,0}^{\g} \rho^{\frac{n}{2}-\g + \lfloor \g \rfloor + 2 + 2\ell} f_{\lfloor \g\rfloor + 2 + 2\ell}\\
    &=\pi_{j,3}^{\g,\e}\pi_{j,2}^{\g}\Pi_{j,0}^{\g}\rho^{\frac{n}{2}-\g + \lfloor \g \rfloor + 2 + 2\ell} f_{\lfloor \g \rfloor + 2 + 2\ell} \mod \rho^{\frac{n}{2} + 2-\g +2j} \ceven.
  \end{align*}
  Using \eqref{eq:laplace-on-rho}, we compute
  \begin{align*}
    \pi_{j,2}^{\g} &= (-1)^{j-\lfloor \g/2 \rfloor - 1 - \ell}\prod_{m=0}^{\lfloor \g/2 \rfloor + \ell}\left( (\g-2j)^{2} - (\g-2m)^{2} \right)\\
    \pi_{j,3}^{\g,\e} &= \e \prod_{m=\lfloor \g \rfloor - j-1, m \neq j}^{\lfloor \g \rfloor}\left( (\g-2j)^{2} - (\g - 2m)^{2} \right).
  \end{align*}
  We lastly record
  \[
    \tilde b_{2j} ^{\e} = \e  \prod_{m=0}^{j-1}\left( (\g-2j)^{2} - (\g-2m)^{2} \right)\prod_{m=\lfloor \g \rfloor -j-1,m \neq j}^{\lfloor \g \rfloor}\left( (\g-2j)^{2} - (\g-2m)^{2} \right).
  \]
  At last, it is then easy to compute
  \begin{align*}
    \lim_{\e\to0} \frac{ \pi_{j,3}^{\g,\e} \pi_{j,2}^{\g} \pi_{j}^{\g}}{b_{2j}^{\e} \pi_{j,1}^{\g}} &=
    \frac{\prod_{m=\lfloor \g \rfloor - j -1,m\neq j}^{\lfloor \g \rfloor} \left( (\g-2j)^{2} - (\g-2m)^{2} \right)}{\prod_{m=0}^{j-1}\left( (\g-2j)^{2} - (\g-2m)^{2} \right)}\\
    &\times\frac{\prod_{m=0}^{\lfloor \g /2 \rfloor + \ell }\left( (\g-2j)^{2} - ( \g -2m)^{2} \right)}{\prod_{m=\lfloor \g \rfloor - j - 1, m \neq j}^{\lfloor \g \rfloor}\left( (\g-2j)^{2} - (\g-2m)^{2} \right)}\\
    &\times \frac{\prod_{m=0}^{j-1} \left( (\g-2j)^{2} - (\g-2m)^{2} \right)}{\prod_{m=0}^{\lfloor \g /2 \rfloor + \ell}\left( (\g-2j)^{2} - (\g-2m)^{2} \right)}\\
    &\times \frac{\prod_{m=0}^{\lfloor \g \rfloor - j - 1}\left( (\g-2j)^{2} - (\g-2j-2-2m)^{2} \right)}{\prod_{m=0}^{\lfloor \g \rfloor - j - 1}\left( (\g-2j)^{2} - (\g-2j-2-2m)^{2} \right)}\\
    &=1.
  \end{align*}
  This is enough to conclude
  \begin{align*}
    &\rho^{-\frac{n}{2} + \g - 2j}\Pi_{j}^{\g}(\rho^{\frac{n}{2}-\g + \lfloor \g \rfloor + 2 + 2\ell} f_{\lfloor \g \rfloor + 2 + 2 \ell})|_{\rho=0}\\
    &- \rho^{-\frac{n}{2} + \g - 2j} \Pi_{j}^{\g} \left( \rho^{\frac{n}{2}-\g+2j} B_{2j}^{2\g} \left(  \rho^{\lfloor \g \rfloor + 2 + 2 \ell} \right)f_{\lfloor \g \rfloor + 2 + 2\ell} \right)|_{\rho=0}\\
    &=0
  \end{align*}
  and therefore
  \begin{align*}
    \rho^{-\frac{n}{2} + \g - 2j} \Pi_{j}^{\g} \rho^{\frac{n}{2}-\g}f |_{\rho=0} &= \pi_{j}^{\g} B_{2j}^{2\g}f\\
    \rho^{-\frac{n}{2} + \g - 2j+1} \p_{\rho} \Pi_{j}^{\g} \rho^{\frac{n}{2}-\g}f |_{\rho=0} &= (\frac{n}{2}-\g+2j)\pi_{j}^{\g} B_{2j}^{2\g}f,
  \end{align*}
  which is what we wanted to show.
  Then
  \begin{align*}
    \rho^{1-n}\p_{\rho}(\rho^{\frac{n}{2}-\g}U_{j}) \Pi_{j}(\rho^{\frac{n}{2}-\g}(V-\tilde V))|_{\rho=0}&= \pi_{j}^{\g} (\frac{n}{2}+\g-2j) B_{2\g-2j}^{2\g}U B_{2j}^{2\g}(V - \tilde V)\\
    \rho^{1-n}(\rho^{\frac{n}{2}-\g}U_{j}) \p_{\rho} \Pi_{j}(\rho^{\frac{n}{2}-\g}(V-\tilde V))|_{\rho=0} &= \pi_{j}^{\g} (\frac{n}{2} - \g + 2j) B_{2\g- 2j}^{2\g}U B_{2j}^{2\g}(V - \tilde V).
  \end{align*}
  Consequently,
  \begin{align*}
    &-\sum_{j = \lfloor \g / 2 \rfloor + 1 }^{\lfloor \g \rfloor} \int_{M}\rho^{1-n}\p_{\rho}(\rho^{\frac{n}{2}-\g}U_{j}) \Pi^{\g}_{j}(\rho^{\frac{n}{2}-\g}(V-\tilde V))\big|_{\rho=0} dvol_{g} \\
    &+ \sum_{j=\lfloor \g/2 \rfloor + 1}^{\lfloor \g \rfloor} \int_{M} \rho^{1-n}\rho^{\frac{n}{2}-\g}U_{j} \p_{\rho}\Pi^{\g}_{j}(\rho^{\frac{n}{2}-\g}(V-\tilde V))\big|_{\rho=0} dvol_{g}\\
    &= \sum_{j=\lfloor \g /2 \rfloor + 1}^{\lfloor \g \rfloor} \int_{M} \pi_{j}^{\g} (4j-2\g) B_{2\g-2j}^{2\g} U B_{2j}^{2\g}(V - \tilde V) dvol_{g}.
  \end{align*}
  We have by Theorem \ref{thm:boundary-to-fractional-operator} that
  \[
    B^{2\g}_{2j}\tilde V = d_{\g,\lfloor \g \rfloor -j} P_{2j - \g} B_{2\g - 2j}^{2\g}V
  \]
  and so
  \begin{align*}
    &\sum_{j=\lfloor \g /2 \rfloor + 1}^{\lfloor \g \rfloor} \int_{M} \pi_{j}^{\g} (4j-2\g) B_{2\g-2j}^{2\g} U B_{2j}^{2\g}(V - \tilde V) dvol_{g}\\
    &=\sum_{j=\lfloor \g/2 \rfloor + 1}^{\lfloor \g \rfloor} \pi_{j}^{\g}(4j-2\g)\int_{M}   B_{2\g-2j}^{2\g}U  B_{2j}^{2\g}V  dvol_{g}\\
    &- \sum_{j=\lfloor \g /2 \rfloor + 1}^{\lfloor \g \rfloor} d_{\g,\lfloor \g \rfloor - j}\pi_{j}^{\g}(4j-2\g)\int_{M}   B_{2\g - 2j}^{2\g}U P_{2j - \g}B_{2\g-2j}^{2\g} V    dvol_{g}.
  \end{align*}

  Putting together Cases 1 and 2, we conclude
  \begin{align*}
    A_{2}  &=-\sum_{j = 0 }^{\lfloor \g \rfloor} \int_{M}\rho^{1-n}\p_{\rho}(\rho^{\frac{n}{2}-\g}U_{j}) \Pi^{\g}_{j}(\rho^{\frac{n}{2}-\g}(V-\tilde V))\big|_{\rho=0} dvol_{g} \\
    &+ \sum_{j=0}^{\lfloor \g \rfloor} \int_{M} \rho^{1-n}\rho^{\frac{n}{2}-\g}U_{j} \p_{\rho}\Pi^{\g}_{j}(\rho^{\frac{n}{2}-\g}(V-\tilde V))\big|_{\rho=0} dvol_{g}\\
    &=\sum_{j=0}^{\lfloor \g/2 \rfloor} \pi_{j}^{\g}(2\g-4j)\int_{M}   B_{2j}^{2\g}U  B_{2\g-2j}^{2\g}V  dvol_{g}\\
    &- \sum_{j=0}^{\lfloor \g/2 \rfloor} c_{\g,j}\pi_{j}^{\g}(2\g-4j)\int_{M}   B_{2j}^{2\g}U P_{\g-2j}B_{2j}^{2\g} V    dvol_{g}\\
    &+\sum_{j=\lfloor \g/2 \rfloor + 1}^{\lfloor \g \rfloor} \pi_{j}^{\g}(4j-2\g)\int_{M}   B_{2\g-2j}^{2\g}U  B_{2j}^{2\g}V  dvol_{g}\\
    &- \sum_{j=\lfloor \g /2 \rfloor + 1}^{\lfloor \g \rfloor} d_{\g,\lfloor \g \rfloor - j}\pi_{j}^{\g}(4j-2\g)\int_{M}   B_{2\g - 2j}^{2\g}U P_{2j - \g}B_{2\g-2j}^{2\g} V    dvol_{g}.
  \end{align*}
  To compute the constants, let
  \begin{align*}
    \sigma_{j,\g} &=
    \begin{cases}
      2^{-n}\pi_{j}^{\g}(2\g-4j)& j = 0,\ldots, \lfloor \g /2 \rfloor\\
      2^{-n}\pi_{j}^{\g}(4j-2\g)& j =\lfloor \g /2 \rfloor + 1,\ldots, \lfloor \g \rfloor
    \end{cases}\\
    \varsigma_{j,\g} &=
    \begin{cases}
      2^{-n}c_{\g,j}\pi_{j}^{\g}(4j-2\g) & j = 0,\ldots, \lfloor \g /2 \rfloor\\
      2^{-n}d_{\g,\lfloor \g \rfloor - j}\pi_{j}^{\g}(2\g-4j) & j =\lfloor \g /2 \rfloor + 1,\ldots, \lfloor \g \rfloor
    \end{cases},
  \end{align*}
  recall
  \begin{align*}
    c_{\gamma,j}&=2^{2j-\g} \frac{\Gamma(2j-\g)}{\Gamma(\g-2j)}\\
    d_{\gamma,j}&=2^{2j+[\g]-\lfloor \g \rfloor} \frac{\Gamma(2j+[\g]-\lfloor \g \rfloor)}{\Gamma(\lfloor \g \rfloor-2j-[\g])}
  \end{align*}
  and compute
  \begin{align*}
    \pi_{j}^{\g}&= \prod_{m=0}^{j-1} \left(-(\g-2j)^{2} + (\g-2m)^{2} \right) \prod_{m=0}^{\lfloor \g \rfloor - j -1} \left(-(\g-2j)^{2} + (\g-2j-2-2m)^{2} \right)\\
    &= \frac{\Gamma(\g-j+1)}{\Gamma(\g-2j+1)} \Gamma(j+1) \frac{\Gamma(j+\lfloor \g \rfloor - \g +1)}{\Gamma(2j-\g+1)}\Gamma(\lfloor \g \rfloor - j + 1).
  \end{align*}
  Putting this together, we may conclude the desired identity.
\end{proof}

\subsection*{Proof of Theorem \ref{thm:dirichlet-form-symmetry}}

We aim to show that
\begin{align*}
  \mathcal{Q}_{2\gamma}(U,V):=&\int_{X}U L_{2k} V \rho^{1-2[\gamma]}dvol_{\rho^2 g_+} \\
  &- \sum_{j=0}^{\lfloor \g/2 \rfloor} \sigma_{j,\g}  \int_{M} B_{2j}^{2\g}U  B_{2\g-2j}^{2\g}V  dvol_{g} - \sum_{j=\lfloor \g/2 \rfloor + 1}^{\lfloor \g \rfloor} \sigma_{j,\g} \int_{M}   B_{2\g-2j}^{2\g}U  B_{2j}^{2\g}V  dvol_{g}\\
  =&\int_{X}V L_{2k} U \rho^{1-2[\gamma]}dvol_{\rho^2 g_+} \\
  &- \sum_{j=0}^{\lfloor \g/2 \rfloor} \sigma_{j,\g}  \int_{M} B_{2j}^{2\g}V  B_{2\g-2j}^{2\g}U  dvol_{g} - \sum_{j=\lfloor \g/2 \rfloor + 1}^{\lfloor \g \rfloor} \sigma_{j,\g} \int_{M}   B_{2\g-2j}^{2\g}V  B_{2j}^{2\g}U  dvol_{g}\\
  &=\mathcal{Q}_{2\g}(V,U).
\end{align*}
Let $\tilde U$ and $\tilde V$ be as above.
By Theorem \ref{thm:main-integral-identity}, there holds
\begin{align*}
  \int_{X} U L_{2k}V \cdot \rho^{1- 2[\g]} dvol_{\rho^2 g_+}  &= \int_{X} \rho^{\frac{n}{2}-\g} (U - \tilde U) L_{2k}^{+}( \rho^{\frac{n}{2}-\g} (V-\tilde V)) dvol_{g} + A_{2}(U,V)\\
  &=: A_{1}(U,V) + A_{2}(U,V),
\end{align*}
where
\begin{align*}
  A_{2}(U,V) &:= \sum_{j=0}^{\lfloor \g/2 \rfloor} \sigma_{j,\g}  \int_{M} B_{2j}^{2\g}U  B_{2\g-2j}^{2\g}V  dvol_{g}+\sum_{j=\lfloor \g/2 \rfloor + 1}^{\lfloor \g \rfloor} \sigma_{j,\g} \int_{M}   B_{2\g-2j}^{2\g}U  B_{2j}^{2\g}V  dvol_{g}\\
  &+ \sum_{j=0}^{\lfloor \g/2 \rfloor} \varsigma_{j,\g} \int_{M}   B_{2j}^{2\g}U P_{\g-2j}B_{2j}^{2\g} V    dvol_{g} + \sum_{j=\lfloor \g /2 \rfloor + 1}^{\lfloor \g \rfloor} \varsigma_{j,\g} \int_{M}   B_{2\g - 2j}^{2\g}U P_{2j - \g}B_{2\g-2j}^{2\g} V    dvol_{g}.
\end{align*}
Since $\rho^{\frac{n}{2}-\g}(U - \tilde U)$ and $\rho^{\frac{n}{2}-\g}(V-\tilde V)$ decay sufficiently fast at $\rho=0$, we may use the self-adjointness of $\Delta_{+}$ to conclude
\[
  A_{1}(U,V) = A_{1}(V,U).
\]
To show $A_{2}(U,V) = A_{2}(V,U)$, we use the self-adjointness of $P_{\g-2j}$ for $j = 0,\ldots, \lfloor \g /2 \rfloor$ and $P_{2j - \g}$ for $j = \lfloor \g/2 \rfloor + 1, \ldots, \lfloor \g \rfloor$ to obtain
\begin{align*}
  A_{2}(U,V) &= \sum_{j=0}^{\lfloor \g/2 \rfloor} \sigma_{j,\g}  \int_{M} B_{2j}^{2\g}U  B_{2\g-2j}^{2\g}V  dvol_{g}+\sum_{j=\lfloor \g/2 \rfloor + 1}^{\lfloor \g \rfloor} \sigma_{j,\g} \int_{M}   B_{2\g-2j}^{2\g}U  B_{2j}^{2\g}V  dvol_{g}\\
  &+ \sum_{j=0}^{\lfloor \g/2 \rfloor} \varsigma_{j,\g} \int_{M}   B_{2j}^{2\g}V P_{\g-2j}B_{2j}^{2\g} U    dvol_{g} + \sum_{j=\lfloor \g /2 \rfloor + 1}^{\lfloor \g \rfloor} \varsigma_{j,\g} \int_{M}   B_{2\g - 2j}^{2\g}V P_{2j - \g}B_{2\g-2j}^{2\g} U    dvol_{g}.
\end{align*}

Using these observations, we get
\begin{align*}
  \mathcal{Q}_{2\gamma}(U,V):=&\int_{X}U L_{2k} V \rho^{1-2[\gamma]}dvol_{\rho^2 g_+} \\
  &- \sum_{j=0}^{\lfloor \g/2 \rfloor} \sigma_{j,\g}  \int_{M} B_{2j}^{2\g}U  B_{2\g-2j}^{2\g}V  dvol_{g}-\sum_{j=\lfloor \g/2 \rfloor + 1}^{\lfloor \g \rfloor} \sigma_{j,\g} \int_{M}   B_{2\g-2j}^{2\g}U  B_{2j}^{2\g}V  dvol_{g}\\
  &= A_{1}(U,V) + A_{2}(U,V)\\
  &- \sum_{j=0}^{\lfloor \g/2 \rfloor} \sigma_{j,\g}  \int_{M} B_{2j}^{2\g}U  B_{2\g-2j}^{2\g}V  dvol_{g}-\sum_{j=\lfloor \g/2 \rfloor + 1}^{\lfloor \g \rfloor} \sigma_{j,\g} \int_{M}   B_{2\g-2j}^{2\g}U  B_{2j}^{2\g}V  dvol_{g}\\
  &= A_{1}(U,V)\\
  &+ \sum_{j=0}^{\lfloor \g/2 \rfloor} \sigma_{j,\g}  \int_{M} B_{2j}^{2\g}U  B_{2\g-2j}^{2\g}V  dvol_{g}+\sum_{j=\lfloor \g/2 \rfloor + 1}^{\lfloor \g \rfloor} \sigma_{j,\g} \int_{M}   B_{2\g-2j}^{2\g}U  B_{2j}^{2\g}V  dvol_{g}\\
  &+ \sum_{j=0}^{\lfloor \g/2 \rfloor} \varsigma_{j,\g} \int_{M}   B_{2j}^{2\g}V P_{\g-2j}B_{2j}^{2\g} U    dvol_{g} + \sum_{j=\lfloor \g /2 \rfloor + 1}^{\lfloor \g \rfloor} \varsigma_{j,\g} \int_{M}   B_{2\g - 2j}^{2\g}V P_{2j - \g}B_{2\g-2j}^{2\g} U    dvol_{g}\\
  &- \sum_{j=0}^{\lfloor \g/2 \rfloor} \sigma_{j,\g}  \int_{M} B_{2j}^{2\g}U  B_{2\g-2j}^{2\g}V  dvol_{g}-\sum_{j=\lfloor \g/2 \rfloor + 1}^{\lfloor \g \rfloor} \sigma_{j,\g} \int_{M}   B_{2\g-2j}^{2\g}U  B_{2j}^{2\g}V  dvol_{g}\\
  &= A_{1}(V,U)\\
  &+ \sum_{j=0}^{\lfloor \g/2 \rfloor} \sigma_{j,\g}  \int_{M} B_{2j}^{2\g}V  B_{2\g-2j}^{2\g}U  dvol_{g}+\sum_{j=\lfloor \g/2 \rfloor + 1}^{\lfloor \g \rfloor} \sigma_{j,\g} \int_{M}   B_{2\g-2j}^{2\g}V  B_{2j}^{2\g}U  dvol_{g}\\
  &+ \sum_{j=0}^{\lfloor \g/2 \rfloor} \varsigma_{j,\g} \int_{M}   B_{2j}^{2\g}V P_{\g-2j}B_{2j}^{2\g} U    dvol_{g} + \sum_{j=\lfloor \g /2 \rfloor + 1}^{\lfloor \g \rfloor} \varsigma_{j,\g} \int_{M}   B_{2\g - 2j}^{2\g}V P_{2j - \g}B_{2\g-2j}^{2\g} U    dvol_{g}\\
  &- \sum_{j=0}^{\lfloor \g/2 \rfloor} \sigma_{j,\g}  \int_{M} B_{2j}^{2\g}V  B_{2\g-2j}^{2\g}U  dvol_{g}-\sum_{j=\lfloor \g/2 \rfloor + 1}^{\lfloor \g \rfloor} \sigma_{j,\g} \int_{M}   B_{2\g-2j}^{2\g}V  B_{2j}^{2\g}U  dvol_{g}\\
  &= A_{1}(V,U) + A_{2}(V,U)\\
  &- \sum_{j=0}^{\lfloor \g/2 \rfloor} \sigma_{j,\g}  \int_{M} B_{2j}^{2\g}V  B_{2\g-2j}^{2\g}U  dvol_{g}-\sum_{j=\lfloor \g/2 \rfloor + 1}^{\lfloor \g \rfloor} \sigma_{j,\g} \int_{M}   B_{2\g-2j}^{2\g}V  B_{2j}^{2\g}U  dvol_{g}\\
  &= \mathcal{Q}_{2\g}(V,U),
\end{align*}
which is what we wanted to show.
\qed

\subsection*{Proof of Theorem \ref{th1.5}}

At last, we now prove the higher order Sobolev trace inequalities recorded in Theorem \ref{th1.5}.

Let $\tilde U$ be as above.
By Theorem \ref{thm:main-integral-identity}, there holds
\begin{align*}
  \int_{X} U L_{2k}U \cdot \rho^{1- 2[\g]} dvol_{\rho^2 g_+}  &= \int_{X} \rho^{\frac{n}{2}-\g} (U - \tilde U) L_{2k}^{+}( \rho^{\frac{n}{2}-\g} (U-\tilde U)) dvol_{g_{+}} + A_{2}(U,U)\\
  &=: A_{1}(U,U) + A_{2}(U,U),
\end{align*}
where
\begin{align*}
  A_{2}(U,U) &:= \sum_{j=0}^{\lfloor \g/2 \rfloor} \sigma_{j,\g}  \int_{M} B_{2j}^{2\g}U  B_{2\g-2j}^{2\g}U  dvol_{g}-\sum_{j=\lfloor \g/2 \rfloor + 1}^{\lfloor \g \rfloor} \sigma_{j,\g} \int_{M}   B_{2\g-2j}^{2\g}U  B_{2j}^{2\g}U  dvol_{g}\\
  &+ \sum_{j=0}^{\lfloor \g/2 \rfloor} \varsigma_{j,\g} \int_{M}   B_{2j}^{2\g}U P_{\g-2j}B_{2j}^{2\g} U    dvol_{g} + \sum_{j=\lfloor \g /2 \rfloor + 1}^{\lfloor \g \rfloor} \varsigma_{j,\g} \int_{M}   B_{2\g - 2j}^{2\g}U P_{2j - \g}B_{2\g-2j}^{2\g} U    dvol_{g}.
\end{align*}
It follows that
\begin{align*}
  \mathcal{E}_{2\g}(U) &:= \mathcal{Q}_{2\gamma}(U,U)\\
  &:=\int_{X}U L_{2k} U \rho^{1-2[\gamma]}dvol_{\rho^2 g_+} \\
  &- \sum_{j=0}^{\lfloor \g/2 \rfloor} \sigma_{j,\g}  \int_{M} B_{2j}^{2\g}U  B_{2\g-2j}^{2\g}U  dvol_{g}+\sum_{j=\lfloor \g/2 \rfloor + 1}^{\lfloor \g \rfloor} \sigma_{j,\g} \int_{M}   B_{2\g-2j}^{2\g}U  B_{2j}^{2\g}U  dvol_{g}\\
  &= A_{1}(U,U)\\
  &+ \sum_{j=0}^{\lfloor \g/2 \rfloor} \varsigma_{j,\g} \int_{M}   B_{2j}^{2\g}U P_{\g-2j}B_{2j}^{2\g} U    dvol_{g} + \sum_{j=\lfloor \g /2 \rfloor + 1}^{\lfloor \g \rfloor} \varsigma_{j,\g} \int_{M}   B_{2\g - 2j}^{2\g}U P_{2j - \g}B_{2\g-2j}^{2\g} U    dvol_{g}.
\end{align*}
To conclude $A_{1}(U,U) \geq 0$, we appeal to Proposition \ref{prop:sepctral-assumption-proposition} to use $\la_{1}(L_{2k}) > 0$.
Moreover, $A_{1}(U,U) = 0$ if and only if $U = \tilde U$.
This is enough to conclude the result.

\qed
\section{On the Spectral Assumption}
\label{sec:spectral-assumption}

In this section we provide a characterization of the boundedness of $\mathcal{E}_{2\g}$ on $C_{\vec f, \vec \psi}^{2\g}$ in terms of $\la_{1}(-\Delta_{+})$, for given boundary data $\vec f, \vec \psi$.
This was obtained by Escobar in \cite{MR1283876} for $\g = \frac{1}{2}$ and Case in \cite{MR3619870} for $\g \in (0,1) \cup (1,2)$.

\begin{proposition}
  \label{prop:sepctral-assumption-proposition}
  Given $\g \in (0,\frac{n}{2}) \setminus \N$, let $(X,M,g_{+})$ be a Poincar\'e-Einstein manifold so that
  \begin{equation}
    \frac{n^{2}}{4} - (2\ell-\g)^{2} \notin \sigma_{pp}(\Delta_{g_{+}})\quad  \text{ for }\quad 0 \leq \ell \leq \lfloor \g \rfloor.
    \label{eq:point-spectrum-assumption}
  \end{equation}
  Let $\rho$ be an $\g$-admissible defining function and fix $\vec f, \vec \psi$.
  Letting $\mcE_{2\g}$ denote the energy corresponding to $L_{2k}$, there holds:
  \begin{equation}
    \inf_{U \in \mcC_{\vec f, \vec \psi}^{2\g}} \mcE_{2\g}(U) > -\oo
    \label{eq:energy-bounded-below}
  \end{equation}
  iff
  \begin{equation}
    \la_{1}(L_{2k}) > 0 .
    \label{eq:spectral-lower-bound-L2k}
  \end{equation}
  Moreover, $\la_{1}(-\Delta_{g_{+}}) > \frac{n^{2}}{4} - (2\lfloor \g \rfloor - \g)^{2}$ implies \eqref{eq:spectral-lower-bound-L2k}, and \eqref{eq:spectral-lower-bound-L2k} implies \eqref{eq:point-spectrum-assumption}.
\end{proposition}

\begin{proof}

  We follow the approach of the proofs of Propositions 5.1 and 5.3 in \cite{MR3619870}.
  By definition,
  \[
    \la_{1}(L_{2k,\rho}) = \inf\left\{ \mcE_{2\g}(V): V \in \mcC_{\vec0,\vec0}^{\g}, \int_{X} V^{2} \rho^{1-2[\g]}dvol_{\rho^{2} g_{+}} = 1. \right\}.
  \]
  Using the factorization
  \[
    L_{2k} = \rho^{-\frac{n}{2} + \g - 2k} \circ \prod_{j=0}^{\lfloor \g \rfloor} \left( -\tilde \Delta_{+} + \frac{(2j - \g)^{2}}{4} \right) \circ \rho^{\frac{n}{2} - \g},
  \]
  we have that $\la_{1}(-\Delta_{g_{+}})$ implies \eqref{eq:spectral-lower-bound-L2k}, and that \eqref{eq:spectral-lower-bound-L2k} implies \eqref{eq:point-spectrum-assumption}.

  Now, we will let $\vec 0$ denote the vector of suitable size consisting only of the 0 functions.
  Let $U \in \mcC_{\vec f, \vec \psi}^{2\g}$ be such that $\mcC_{\vec f, \vec \psi}^{2\g} = U + C_{\vec 0,\vec 0}^{2\g}$, let $V \in \mcC_{\vec 0, \vec 0}^{2\g}$, let $t \in \R$ and record
  \begin{equation}
    \mcE_{2\g}(U + tV) = t^{2} \mcE_{2\g}(V) + 2 t \mcQ_{2\g}(U,V) + \mcE_{2\g}(U).
  \end{equation}
  Note that
  \begin{align*}
    \mcE_{2\g}(V) &= \int_{X} V L_{2k,\rho} V \rho^{1-2[\g]} dvol_{\rho^{2}g_{+}}\\
    \mcQ_{2\g}(U,V) &= \int_{X} V L_{2k,\rho} U \rho^{1-2[\g]} dvol_{\rho^{2}g_{+}}.
  \end{align*}
  Now, assuming \eqref{eq:spectral-lower-bound-L2k} fails, we have $\mcE_{2\g}(V_{0}) < 0$ for some $V_{0} \in \mcC_{\vec0,\vec0}^{2\g}$.
  This implies \[\mcE_{2\g}(U + tV_{0}) \to -\oo\] as $t \to \oo$ and so \eqref{eq:energy-bounded-below} fails.

  Next, assume \eqref{eq:spectral-lower-bound-L2k} holds.
  Observe then that
  \begin{align*}
    \mcE_{2\g}(U + V)& \geq \la_{1}(L_{2k,\rho})  \int_{X}V^{2}\rho^{1-2[\g]}dvol_{\rho^{2} g_{+}}\\
    &- 2  \left( \int_{X} V^{2}\rho^{1-2[\g]}dvol_{\rho^{2} g_{+}} \right)^{1/2} \left( \int_{X} (L_{2k,\rho}U)^{2}\rho^{1-2[\g]}dvol_{\rho^{2} g_{+}} \right)^{2} + \mcE_{2\g}(U),
  \end{align*}
  which is uniformly bounded below for all $V \in \mcC_{\vec 0, \vec 0}^{2\g}$ and therefore \eqref{eq:energy-bounded-below} holds.
\end{proof}

\end{document}